\documentclass[11pt]{amsart}

\usepackage[a4paper,inner=3.6cm,outer=3.6cm,top=3.7cm,bottom=3.7cm]{geometry}

\usepackage{amssymb,amsmath,mathtools}
\usepackage{graphicx}
\usepackage{subfigure}
\usepackage{ upgreek }
\usepackage{tikz, tikz-cd}
\usepackage[pdftex]{hyperref}
\usetikzlibrary{matrix}
\usepackage[all]{xy}

\newcommand{\bC}{{\mathbb C}}

\newcommand{\bR}{{\mathbb R}}

\newcommand{\cC}{\mathcal C}

\newcommand{\cS}{\mathcal S}

\newcommand{\cM}{\mathcal M}
\newcommand{\cO}{\mathcal O}

\newcommand{\cL}{\mathcal L}
\newcommand{\cR}{\mathcal R}

\def\co{\colon\thinspace}

\newcommand{\bw}{\mathbf w}
\newcommand{\bH}{\mathbf H}
\usepackage{comment}

\newtheorem{thm}{Theorem}[section]
\newtheorem{cor}[thm]{Corollary}
\newtheorem{lem}[thm]{Lemma}
\newtheorem{lemma}[thm]{Lemma}
\newtheorem{prop}[thm]{Proposition}
\newtheorem{defin}[thm]{Definition}
\newtheorem{def-lem}[thm]{Definition-Lemma}
\newtheorem{conj}[thm]{Conjecture}

\theoremstyle{remark}

\newtheorem{rem}[thm]{Remark}

\newtheorem{example}[thm]{Example}

\newcommand\abs[1]{|#1|}

\begin{document}

\title[Super-rigid skeleta]{Super-rigidity of certain skeleta using relative symplectic cohomology}
\author{Dmitry Tonkonog}
\address{(DT) Harvard University}
\address{(DT) University of California, Berkeley}
\author{Umut Varolgunes}
\address{(UV) Stanford University}
\begin{abstract}
	This article uses relative symplectic cohomology, recently studied by the second author, to understand rigidity phenomena for compact subsets of symplectic manifolds. As an application, we consider a symplectic crossings divisor in a Calabi-Yau symplectic manifold $M$ whose complement is a Liouville manifold. We show that, for a carefully chosen Liouville structure, the skeleton as a subset of $M$ exhibits strong rigidity properties akin to super-heavy subsets of Entov-Polterovich. 
	
	Along the way, we expand the toolkit of relative symplectic cohomology by introducing products and units. We also develop what we call the contact Fukaya trick, concerning the behaviour of relative symplectic cohomology of subsets with contact type boundary under adding a Liouville collar.
\end{abstract}
\maketitle

\section{Introduction}

\subsection{Motivation}
%
%

Let $M$ be a symplectic manifold. Let us recall some standard definitions (see e.g.~\cite{EnPo09}). Two subsets $K,S\subset M$ are called {\it displaceable} from each other if there exists a Hamiltonian diffeomorphism $\phi\co M\to M$ such that $\phi(K)\cap S=\emptyset$. Otherwise, $S$ is called non-displaceable from $K$ and vice versa. A subset $K\subset M$ is called (non-) displaceable if it is (non-) displaceable from itself. If $K,S\subset M$ are not displaceble from each other via a symplectomorphism, then we call them {\it strongly non-displaceable}.

Subsets $K,S\subset M$ are called {\it stably displaceable} if $K\times S^1$ is displaceable from $S\times S^1$ as subsets of the symplectic manifold $M\times T^*S^1$,  with the product symplectic form, where $S^1$ represents the zero section. Otherwise, they are called {\it stably non-displaceable}, and $K\subset M$ is called stably (non-) displaceable if its stably (non-) displaceable from itself.

When the subset $K\subset M$ is a Lagrangian submanifold,  Lagrangian Floer theory, see e.g.~\cite{FO3Book}, presents a powerful machine for detecting non-displaceability. Entov and Polterovich \cite{EnPo09} introduced the notions of heavy and superheavy sets (using their construction of a partial symplectic quasi-state via spectral invariants \cite{EP06}), giving access to non-displaceability questions for general compact subsets. These two techniques were combined in the seminal work \cite{FO3memoir}. 

In this paper, we approach these questions via another set of tools which uses recent advancements in Floer theory. Let us mention Corollory \ref{mainresult} as our main Floer-theory-free result. Some of our readers might want to start reading from this statement and work their way backwards (and then hopefully forwards) in the introduction.

\subsection{Symplectic-cohomological visibility}
Let $(M,\omega)$ be a closed symplectic manifold. Recall the definitions of the Novikov ring
$$
\Lambda_{\ge 0}=\left\{\sum_{i=1}^\infty a_iT^{\omega_i}: a_i\in \mathbb{Q},\ \omega_i\in \bR_{\ge 0},\ \omega_i\to+\infty \right\}
$$
and the Novikov field
$$
\Lambda=\left\{\sum_{i=1}^\infty a_iT^{\omega_i}: a_i\in \mathbb{Q},\ \omega_i\in \bR,\ \omega_i\to+\infty \right\}.
$$


We approach non-displaceability questions for general compact subsets using {\it relative symplectic cohomology}, recently studied by the second author \cite{VaThesis,Va18}.
The reader might benefit from taking a quick look at Section \ref{ssremindersh} below, where we provided a summary of the construction. For a compact subset $K\subset M$, let $SH_M(K,\Lambda)$ be the $\Lambda$-vector space as defined in Equation (\ref{eqdefsh}).


\begin{defin}
	A compact subset $K$ is called $SH$\textbf{-invisible} if $SH_M(K,\Lambda)=0$, and $SH$\textbf{-visible} otherwise.
\end{defin}

Below we list some properties of this notion. 
Recall that a {\it compact domain} in $M$ is a compact codimension-zero submanifold with boundary. For $K\subset M$, we say that a sequence of compact domains $D_1,D_2,\ldots $ {\it approximate} $K$, if $\bigcap D_i=K$
and $D_{i+1}\subset int(D_i)$ , for all $i\geq 1$.

\begin{thm}
	
	\label{thm:visib_properties}
	The following properties hold.

\begin{enumerate}
	\item $M\subset M$ is $SH$-visible.
	\item SH-visibility is invariant under symplectomorphisms. 
	\item If a compact subset $K$ is stably displaceable, then it is $SH$-invisible.
	\item If a compact subset $K$ is $SH$-invisible, then any compact subset $K'\subset K$ is also $SH$-invisible.
	\item If $K$ and $K'$ are $SH$-invisible compact subsets which can be approximated by compact domains  $D_1,D_2,\ldots $ and $D_1',D_2',\ldots $ such that for evey $i\geq 1$, $\partial D_i\cap \partial D_i'=\emptyset$, then $K\cup K'$ is also $SH$-invisible. 
	\item Let $K$ be an $SH$-invisible compact domain. Then $K$ does not contain any tautologically unobstructed\footnote{A Lagrangian submanifold $L$ is called tautologically unobstructed if there exists a compatible almost complex structure $J$ such that $L$ bounds no non-constant $J$-holomorphic disks.} orientable Lagrangian submanifold (with a relative Pin structure) in its interior. 
\end{enumerate}
\end{thm}

Items~(1),~(2), (3) and~(5) were proved in \cite{VaThesis}, where (5) is a special case of the Mayer-Vietoris property, for which a more polished reference is \cite{Va18} (specifically Section 4.3). The proof of~(4) and~(6) is one of the main contributions of the present paper. 

\begin{conj}\label{conjecture}
Item (6) in Theorem \ref{thm:visib_properties} can be upgraded to Lagrangian submanifolds admitting bounding cochains (in the sense of \cite{FO3Book}) with nonzero self-Floer cohomology.
\end{conj}

\begin{rem}The main piece missing in proving this conjecture is a systematic treatment of the full Hamiltonian isotopy invariance package (including all higher homotopy coherences between continuation maps and the required bounds on the quantitative contributions of Floer solutions via topological energy) for the Floer theory of Lagrangians equipped with bounding cochains. Such a treatment is possible but is outside of the scope of this paper. We expect every statement we make about tautologically obstructed Lagrangians to be also true for Floer theoretically essential Lagrangians. 
\end{rem}

\begin{defin}
	A compact subset $K$ is called $SH$\textbf{-full} if every compact set contained in its complement $M\setminus K$ is $SH$-invisible.
\end{defin} 

Below is a direct corollary of Theorem~\ref{thm:visib_properties}. 

\begin{cor}
	\label{cor:diagram_vis_and_full}
	The following implications are true.
$$
\begin{array}{cccc}
(i)&
K \text{ is $SH$-visible }&\Rightarrow& K\text{ is stably non-displaceable from itself};
\vspace{0.4em}
\\
(ii)&
K \text{ is $SH$-full }&\Rightarrow& K\text{ is strongly non-displaceable~from any $SH$-visible subset};
\vspace{0.4em}
\\
(iii)&
K \text{ is $SH$-full }&\Rightarrow& 
\begin{array}{c}
K\text{ is strongly non-displaceable~from any}\\
\text{tautologically unobstructed Lagrangian submanifold.}
\end{array}
\end{array}
$$
\end{cor}
\begin{proof}
	We shall refer to the items from Theorem~\ref{thm:visib_properties} simply by their numbers.
The first implication is equivalent to~(3).

Let us prove the second implication.
Suppose $K,S\subset  M$ are compact subsets, $K$ is $SH$-full and $S$ is $SH$-visible. Suppose $K,S$ are  displaceable by a symplectomorphism; by~(2), we may assume that they are actually disjoint. This is a contradiction.

Let us prove the third implication. 
Let $K$ be an $SH$-full subset, $L$ a tautologically unobstructed Lagrangian submanifold, and assume $L$ is disjoint from $K$. Then the closure of some tubular neighbourhood $\overline{U(L)}$ is still disjoint from $K$.
By (ii), $\overline{U(L)}$ is $SH$-invisible. This contradicts (6).
\end{proof}	

\begin{rem}
	The existence of such notions as 
$SH$-fullness, and $SH$-visibility is not surprising. Their behavior is similar to the notions of superheaviness, and heaviness (resp.) of Entov-Polterovich \cite{EnPo09}. We hope the precise relationship between the two frameworks will be explained in a collective effort.
\end{rem}

\begin{rem}
Entov-Polterovich in fact define heaviness and superheaviness with respect to any idempotent in quantum cohomology. Since our main application is to Calabi-Yau manifolds\footnote{This only means $c_1(M)=0$ throughout the paper.} in this paper, where the only possible idempotent is the unit, we only provide a brief remark about the analogous construction in our framework. A priori, we can define for any ideal $I$ of the supercommutative $\Lambda-$algebra $QH(M,\Lambda)=SH_M(M,\Lambda)$, the notions of $SH$-invisibility, $SH$-visibility and $SH$-fullness with respect to $I$, e.g. $K$ is $SH$-invisible for $I$, if the submodule $I\cdot SH_M(K)=0$ and so on. Here we are using that $SH_M(K,\Lambda)$ is naturally a $SH_M(M,\Lambda)$-module via the restriction maps (which is a structure developed in the present paper).   

Nevertheless, let us show that a principal ideal of $QH^{even}(M,\Lambda)$ generated by an idempotent does have a special role in our story. First, note that using the techniques of Sections \ref{sec:prod_unit}, it can be shown that the connecting homomorphisms in the Mayer-Vietoris sequence of \cite{Va18} are in fact $QH(M,\Lambda)$-module homomorphisms. 

Let us use only the module structures over the even part of the quantum cohomology from now on (by restricting scalars). Let $I$ be an ideal of $QH^{even}(M,\Lambda)$. Then, we can multiply the Mayer-Vietoris sequence by this ideal (as a diagram of $QH^{even}(M,\Lambda)$-modules), and ask when the result is still an exact sequence.

We can think of the periodically extended Mayer-Vietoris sequence as a $\mathbb{Z}$-graded chain complex $A^*$ with trivial cohomology. Similarly, we have a short exact sequence of chain complexes $$0\to IA^*\to A^* \to A^*\otimes_{QH^{even}(M,\Lambda)} QH^{even}(M,\Lambda)/I\to 0.$$ It follows from the long exact sequence of cohomologies (see Lemma 2.6 and Corollary 3.3 of \cite{Bergman} for a far reaching generalization) that $IA^*$ has trivial cohomology (what we are after) if $QH^{even}(M,\Lambda)/I$ is flat over $QH^{even}(M,\Lambda)$, in other words if $I$ is pure (see \cite[\href{https://stacks.math.columbia.edu/tag/04PQ}{Section 04PQ}]{stacks-project}). Since $I$ is clearly finitely generated, this is equivalent to $I$ being generated by an idempotent \cite[\href{https://stacks.math.columbia.edu/tag/05KK}{Lemma 05KK}]{stacks-project}. 

In particular, for a given idempotent $a$ in $QH^{even}(M,\Lambda)$, the vector spaces $a\cdot SH_M(K,\Lambda)$ satisfy the same Mayer-Vietoris property. We say $K$ is $a$-$SH$-invisible if $SH_M(K,\Lambda)$ is annihilated by $a$, and define $a$-$SH$-visible and $a$-$SH$-full as before. For example, Corollary \ref{cor:full_is_vis} can be generalized to any such $a$. \qed
\end{rem}


We expect that $SH$-fullness implies $SH$-visibility, but we can only show a slightly weaker statement. 

\begin{cor}\label{cor:full_is_vis}
	If $K$ is $SH$-full, then any compact domain that contains $K$ in its interior is $SH$-visible.
\end{cor}

\begin{proof}
 Consider a compact domain $D$ containing $K$. Then the compact domains 
	$D$ and $\overline{M\setminus D}$ 	satisfy the conditions of Theorem~\ref{thm:visib_properties}, Item~(5). Since the second set is $SH$-invisible, the first one is $SH$-visible by Item (1). 
\end{proof}

\begin{example}
A simple closed curve in $S^2$ is $SH$-full if it divides $M$ into two disks of equal areas, using items (3) and (5) of Theorem~\ref{thm:visib_properties}, and $SH$-invisible otherwise. A tubular neighborhood of a non-contractible simple closed curve in the two torus is $SH$-visible, by item (6) of Theorem~\ref{thm:visib_properties},  but it is not $SH$-full (e.g. by item (iii) of Corollary~\ref{cor:diagram_vis_and_full}).
\end{example}

\subsection{Trivial Liouville cobordisms}
 
 \begin{defin}\label{defneck} A \textbf{neck} in $M$ is a symplectic embedding of a trivial (compact) Liouville cobordism $$\Sigma\times [1  -\alpha,s+\alpha]\subset M$$ for some $s>1$ and $\alpha>0$.  Here $\Sigma$ is a closed manifold equipped with a contact form, and $[1  -\alpha,s+\alpha]\times \Sigma$ is a subset of its symplectization.
The coordinate 
$$r\in [1-\alpha,s+\alpha]$$
is the exponential of the Liouville coordinate. That is, if
$\rho$ is the Lioville coordinate such that $\cL_{\partial/\partial \rho}\omega=\omega$, then
$r=e^\rho$. We, in addition, assume that the hypersurface $\Sigma\times\{1\}$ is separating in $M$. One last requirement is that the periodic orbits of the Reeb vector field on $\Sigma$ should be all transversely non-degenerate.\end{defin}

Given a seperating contact hypersurface $\Sigma$ in $M$, we can talk about its convex and concave fillings which are both compact domains with boundary $\Sigma$. 

If $\Sigma\times [1  -\alpha,s+\alpha]\subset M$ is a neck, we can make the convex filling $D$ of $\Sigma\times \{1\}$ larger by adding $\Sigma\times [1,s]\subset M$ to it and hence making it the convex filling $\tilde{D}$ of $\Sigma\times \{s\}$. There is a similar statement for the concave fillings. We are interested in the question: when is the restriction map $SH^*_M(\tilde{D},\Lambda)\to SH^*_M(D,\Lambda)$ an isomorphism?


\begin{defin}
	Suppose $c_1(M)=0$.
	A contact hypersurface $(\Sigma,\theta)$  is called \textbf{index bounded} if all of its Reeb orbits are contractible inside $M$, and for any integer $k$, the periods of the Reeb orbits of Conley-Zehnder index $k$ are bounded above and below. We say that a neck $\Sigma\times [1  -\alpha,s+\alpha]\subset M$ is index bounded if $\Sigma\times \{1\}$ is index bounded. Similarly, we call a Liouville subdomain of $M$ index bounded if its boundary is index bounded.
\end{defin}

\begin{prop}
	\label{prop:index_bd_collar_invt}
	Assume that $c_1(M)=0$,  $\Sigma\times [1  -\alpha,s+\alpha]\subset M$ is an index bounded neck  and $W\subset M$ is either the convex filling of $\Sigma\times \{1\}$ or the concave filling of $\Sigma\times \{s\}$. Then, there exists an isomorphism
	$$
	SH^*_M(W\cup \Sigma\times [1 ,s],\Lambda)\to SH^*_M(W,\Lambda).
	$$
\end{prop}

\begin{rem}
The isomorphism is explicitly constructed using what we call the contact Fukaya trick. Using the argument in Lemma 4.1.1 of \cite{VaThesis} we can actually show that this is isomorphism is given by the restriction map as we had initially asked for. Proving this would make Section \ref{secfuk} even more technical and this strengthening does not help us in the paper.
\end{rem}

Using the well-known index computations for the Reeb orbits of contact boundaries of ellipsoids in $\mathbb{C}^n$ (e.g. Equation (2-6) from \cite{gutt}), we immediately obtain the following corollary.

\begin{cor}
	Suppose $c_1(M)=0$. Let us take disjoint embeddings of symplectic ellipsoids in $\mathbb{C}^n$ into $M^{2n}$. Then we obtain that their union $D$ is $SH-$invisible and hence that the closure of the complement of $D$ is $SH-$full. 
\end{cor}

\begin{rem}
This is far from the best statement we could prove, but it makes the point, namely that here the images of the ellipsoids do not have to be displaceable, but they are $SH$-invisible. We could prove the statement of the corollary, with the same ease, for convex or concave toric domains $\mathbb{C}^n$ (see Sections 2.2 and 2.3 of \cite{gutt}). The statement is also true for convex domains in $\mathbb{C}^n$ but this requires a slightly modified argument using the fact that convex domains are dynamically convex, rather than index bounded (which might also be true but we do not know it in general). We do not spell this out because the symplectic consequences of $SH$-fullness in this case was already covered by Ishikawa's superheaviness result from \cite{ishikawa}.
\end{rem}



Our next step is to discuss skeleta in symplectic manifolds. In the present paper they will be the main source of examples to which we apply our non-displaceability results.

\subsection{Giroux divisors}
We refer the reader to \cite{tehrani,mclean} for the notion of an {\it SC divisor} $D=\bigcup D_i$ in a symplectic manifold $(M,\omega)$. Briefly, this is a union of cleanly intersecting codimension~two symplectic submanifolds $D_i$, such that the intersection orientations coincide with the symplectic orientations on all intersection strata.

As repopularized by McLean, a consequence of the $h$-principle for open symplectic embeddings of codimension~two is the following proposition.

\begin{prop}[Eliashberg-Mishachev \cite{EM02}, McLean \cite{mclean}] 
\label{prop:D_Stably_disp}	
	Let $D$ be an SC divisor in $M$. Then, $D$ is stably displaceable. \qed
\end{prop}

\begin{rem}
To compare, it follows from Theorem~\ref{thm:visib_properties} that $M\subset M$ is not stably displaceable. One may be interested in the intermediate behaviour: which neighborhoods of $D$ are stably displaceable? This is a hard question; for example, it is not known whether a disk  inside a two-torus enclosing more than half of the total area is stably displaceable. 
\end{rem}

We introduce the following definition.

\begin{defin}\label{defgir}
	A \textbf{Giroux divisor} $D=\bigcup D_i$ is an SC divisor with the property that there exist integers $w_i>0$ and a real number $c>0$ such that 
	$$\sum w_i[D_i]=c\cdot PD[\omega]\ \in\ H_2(M).$$
\end{defin}

Below is a structural result about the complements of Giroux divisors. The result follows from the work of McLean \cite{mclean}, along with a construction we learned from \cite{Gi17}. A sketch proof is given in the Appendix\footnote{Added in first revision: We could actually replace the argument in Appendix with a more elementary one that relies on the relative deRham isomorphism (see Section 2.1.2 of \cite{becker}). This would also allow us to work with SC divisors which satisfy $\sum r_i[D_i]=PD[\omega]\ \in\ H_2(M)$ with real numbers $r_i>0$. We still believe that our argument is useful as it provides a lot more control over the Liouville structure produced by the next proposition.}.

\begin{prop}
	\label{prop:giroux_skeleton}
	Let $D\subset M$ be a Giroux divisor, then there exists a Liouville subdomain $W\subset M\setminus D$ such that:
	\begin{itemize}
		\item $\overline{M\setminus W}$ is stably displaceable,
		
		\item $\overline{M\setminus W}$ deformation retracts onto $D$.
	\end{itemize}
Moreover, if $c_1(M)=0$, one can choose $W$ to be index bounded.
\end{prop}

\begin{rem}
	\label{rem:giroux_kahler}
	Suppose $M$ is a complex projective variety and $D$ is an ample normal crossings divisor. Equip $M$ with the symplectic structure given by the curvature $2$-form of $\mathcal{O}(D)$ for the Chern connection of an appropriate Hermitian metric (using Kodaira embedding theorem). Then, $D$ becomes an example of a Giroux divisor. Taking the canonical section $s$ of $\mathcal{O}(D)$ vanishing at $D$, we can equip $M-D$ with the Liouville form $-d^c log\|s\|$. We note that the construction in the Appendix is a generalization of this classical one.
	
	Now assume $c_1(M)=0$. The complement of $D$ in $M$ is exhausted by Liouville (in fact Weinstein) domains $\|s\|\geq c>0$. 	By Proposition~\ref{prop:D_Stably_disp},
	for a sufficiently large such domain $W$, $M\setminus W$ is stably displaceable; and if in addition $D$ is smooth, then $W$ is also index bounded. 
So $W$ satisfies all conditions in Proposition~\ref{prop:giroux_skeleton}. This statement generalizes to smooth Giroux divisors in general symplectic manifolds with the Liouville structure from the Appendix. Also, see Example \ref{examplediskbundle}.
	
	 If $D$ is not smooth, then Proposition \ref{prop:giroux_skeleton} modifies these standard Liouville structures to a Liouville deformation equivalent ``nice'' Liouville structure in the sense of McLean to achieve index boundedness. \qed
\end{rem}

%
%
%

\begin{defin}\label{defsket}
Assume $c_1(M)=0$, and let $D\subset M$ be a Giroux divisor. By a \textbf{skeleton} of $M\setminus D$ we mean  the skeleton of any index bounded Liouville subdomain $W\subset M\setminus D$ for which $\overline{M\setminus W}$ is stably displaceable.
\end{defin}

\begin{example}\label{examplediskbundle}
	Assume that $D$ is smooth, and $[D]=\frac{1}{\pi} PD([\omega])$. Then
	$$
	M=K\sqcup U_1(D)
	$$
	where $K$ is a skeleton of $M\setminus D$ and $U_1(D)$ is an open unit disk symplectic bundle over $D$ as in Figure~\ref{fig:compl_d}, see Corollary 8 of \cite{Gi17}, cf.~\cite{Bi01} in the K\"ahler case. Then for $0<r_2<r_1<1$, $\hat W=M\setminus U_{r_2}$ is the result of neck attachment to $W=M\setminus U_{r_1}$. 
	
	\begin{figure}[h]
		\includegraphics[]{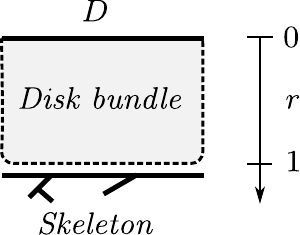}
		\caption{Decomposition of a symplectic manifold into a disk bundle over a smooth symplectic divisor dual to $[\omega]$, and a skeleton.}
		\label{fig:compl_d}
	\end{figure}
\end{example}

Putting the above results together, we obtain the following super-rigidity result for skeleta.

\begin{thm}
	\label{th:sk_superrigid}
	$M$ is a closed symplectic manifold. Assume that $c_1(M)=0$,  $D\subset M$ is a Giroux divisor as in Definition \ref{defgir} and $K\subset M$ is a skeleton of $M\setminus D$ as in Definition \ref{defsket}. Then, $K$ is $SH$-full inside $M$.
\end{thm}

\begin{proof}
	Let $W\subset M\setminus D$ be as in Proposition~\ref{prop:giroux_skeleton}, and $K$ the skeleton of $W$. Let $W_i$ be the image of $W$ under its time-$i$ negative Liouville flow; this gives the nested sequence
	$$
	W=W_0\supset W_1\supset W_2\ldots \supset W_i\supset\ldots\quad \supset K.
	$$
	Recall that $\overline{M\setminus W}$ is stably displaceable. Hence $W$ is $SH$-full. For any $i\geq 1$, $\overline{M\setminus W_i}$ is the result of neck attachment to $\overline{M\setminus W}$, so by Proposition~\ref{prop:index_bd_collar_invt}, $\overline{M\setminus W_i}$ is $SH$-invisible. Using Theorem~\ref{thm:visib_properties}, Item (3) and the definition of the skeleton we obtain the result.
\end{proof}

\begin{cor}\label{mainresult}
	Under the conditions of Theorem~\ref{th:sk_superrigid}:
	\begin{itemize}
		\item $K$ is stably non-displaceable from itself inside $M$;
		\item $K$ is strongly non-displaceable from any tautologically unobstructed Lagrangian submanifold $L
		\subset M$.
	\end{itemize}	
\end{cor}

\begin{proof}
	This immediately follows from Theorem~\ref{th:sk_superrigid} and Corollary~\ref{cor:diagram_vis_and_full}.
\end{proof}

\begin{rem}
Recall that Casals and Murphy \cite{CM19} found examples of smooth affine varieties which are flexible (but are not subcritical). Using the uniruledness considerations by Zhou \cite{zzhou}, see Theorem 5.10, it follows that they do not admit smooth Calabi-Yau compactifications by normal crossings divisors. We do not know if this can be proved using purely algebro-geometric techniques (for example by computing their log Kodaira dimensions). We want to comment on a couple of incomplete strategies to prove this fact using our methods.

Assume that we are in the situation of Remark~\ref{rem:giroux_kahler}:
 $M$ is a complex projective Calabi-Yau and $D$ is an ample normal crossings divisor. Let $K$ be a  skeleton of $M\setminus D$ as in Definition \ref{defsket}.
Our theorem implies that $K$  is stably non-displaceable in $M$.
In particular, it is stably non-displaceable in $W$. If we knew that $W$ was Weinstein, given the expected results on $h$-principles for symplectic embeddings of flexible Weinstein domains, it would follow that $W$ is not flexible. 
Here we also need the fact  that  the product of a flexible Weinstein manifold with $T^*S^1$ is still flexible \cite{Eli17}. Unfortunately, we do not know if the structure on $W$ is Weinstein if $D$ is not smooth. In that case, since the notion of flexibility is currently not defined for Liouville manifolds, we cannot go further with this strategy.

Here is an alternative approach. Our argument shows that $SH_M(W,\Lambda)\neq 0$. We expect that the unpublished results of Borman-Sheridan (currently under preparation with the addition of the second author of the present paper) imply that the Viterbo symplectic cohomology $SH(\widehat{W},\mathbb{Z})\neq 0$. Note the important fact that Viterbo symplectic cohomology is a Liouville deformation invariant. Therefore, we would obtain the desired result from the first paragraph as flexible Weinstein domains have vanishing Viterbo symplectic cohomology.

\qed
\end{rem}

\begin{rem}
We point out that there are examples of Giroux divisors and $W$, which are not even deformation equivalent to a Weinstein domain, see \cite[Proposition~9]{Gi17} for an example\footnote{There is a construction of a connected Giroux divisor in $T^6$ in the second author's blog: \url{https://symplectosaurus.wordpress.com/2020/03/06/roux-divisors}. In a comment to that post Denis Auroux showed that this also gives rise to a smooth and connected Giroux divisor in  $T^6$.} (this is to be contrasted with Donaldson type divisors \cite[Theorems~1,~2]{Gi17}). The nature of the skeleta in these cases is rather mysterious and we are not aware of any rigidity results for them in the literature. Our results apply to these cases equally well.
\end{rem}

\begin{rem} It is possible to prove non-trivial special cases of the second bullet point of Corollary \ref{mainresult} by more elementary means.  Assume that $M$ is a complex projective variety of real dimension $2n$ and $D$ is a smooth ample divisor, let $U=M-D$. Then it can be shown using weak and hard Lefschetz theorems, and the Gysin sequence that $H^n_{prim}(M)\to H^n(U)$ is injective\footnote{Please see this discussion for counter-examples when $D$ is not smooth but irreducible: \url{https://mathoverflow.net/questions/299211/cohomology-groups-of-the-complement-of-an-ample-singular-divisor.}}. Note that here we are using $[\omega]:=PD[D]$ as our Kahler class, and recall that a deRham cohomology class $A\in H^n(M)$ is called primitive if $A\wedge [\omega]=0$.


Take a smooth oriented Lagrangian $L$ in $M$. Assume that it does not intersect the skeleton. This implies that the integral over $L$ of any closed compactly supported form in $U$ (extended by $0$ to $M$) is zero, since we can always find a cohomologous form supported arbitrarily close to the skeleton with the help of the Liouville flow.

Now let $\alpha$ be a Thom form for $L$ inside $M$, meaning for any closed  $n$-form $\beta$ on $M$: $\int_L \beta=\int_M \alpha\wedge\beta$. It is easy to see that $[\alpha]$ is a primitive class using that if it were not, there would have had to be a closed form $\gamma$ such that $\int_M\alpha\wedge\omega\wedge\gamma\neq 0$ by Poincare duality (i.e. existence of Hodge star). This is impossible because $\omega$ vanishes identically on $L$.

Moreover, the image of $[\alpha]$ under $H^n_{prim}(M)\to H^n(U)$ is zero by Poincare duality for $U$. Namely, if it were not, we could find a compactly supported closed $n$-form on $U$ which would pair non-trivially with it. This would be a contradiction to our previous findings. The upshot is that $[L]=0$.

Therefore, a smooth homologically essential Lagrangian can not displaced from the skeleton by a symplectomorphism. In fact being a little more careful we could replace symplectomorphism with a  homeomorphism that preserves $[\omega]$. We feel it is worth exploring how far such a proof (with no reference to $J$-holomorphic curves) can be pushed. For example is there an extension to the case when $M$ is Calabi-Yau, but $D$ is allowed to have normal crossing singularities?\qed
\end{rem}

\begin{rem}
We expect that the techniques of Ishikawa from \cite{ishikawa} also can be used to prove Corollary \ref{mainresult}. For the experts we mention that this would hinge upon some index computations for the constant periodic orbits ``on the divisor". There would also be some technical details for allowing a finite number of non-constant periodic orbits (of multiples of a distance to the divisor type Hamiltonian) which live near the ``boundary" of $K$ and has the right CZ index to affect the spectral number of the unit. We hope that an expert will check the validity of our expectation.
\end{rem}

%

\begin{example}
	\label{ex:two_torus}
Let $M$ be the two-torus and $D$ a point. In this case, one can take a skeleton $K$ of $M\setminus D$ to be the union of a meridian and a longtitude. This is an example falling under Corollary ~\ref{mainresult}. For a non-example, one can take $M=S^2$ and $D$ a point.
\end{example}

%
%

Let us also state the following theorem, which follows immediately from the techniques of this paper.
\begin{thm}
	\label{thm:higher_genus}
Assume that $M$ is a closed surface of genus at least $1$ and let $K$ be the complement of a finite disjoint union of open subsets of $M$ all diffeomorphic to disks. Then $K$ is $SH$-full as a subset of $M$.\qed
\end{thm}

\begin{rem}
That such $K$ is superheavy is also a special case of the main result of \cite{ishikawa}.
\end{rem}

Note that $M$ as in Theorem~\ref{thm:higher_genus} is not Calabi-Yau, but it is aspherical. Hence, we can consider the summand of relative cohomology corresponding to the contractible orbits and use the displaceability of sufficiently small disks along with the contact Fukaya trick to prove the result just the same.

One can also directly show that a smooth closed disk of any size $\mathbb{D}\subset M$ satisfies $SH_M(\mathbb{D},\Lambda) = 0$ (which implies Theorem~\ref{thm:higher_genus} by the Mayer-Vietoris property) without relying on the contact Fukaya trick. This argument still uses the asphericity of $M$ and relies on a topological elimination of non-local contributions to the differential of the relative symplectic cochain complex of $\mathbb{D}$ using positivity of intersection of a Floer solution with a point invariant under the Hamiltonian flows in question. Another ingredient is the integrated maximum principle of \cite{ASei10}. We leave this as a non-trivial exercise to the interested reader.

As Theorem~\ref{thm:higher_genus} suggests (by taking $D$ a finite set of points, and $K\subset M$ any skeleton of $M\setminus D$), Corollary \ref{mainresult} should be more general, in particular, should hold in many non-Calabi-Yau cases. On the other hand, its failure is in some sense even more interesting. We plan to explore both directions in future work.

\subsection{Algebraic structures for relative symplectic cohomology}
To prove Items~(4) and~(6) from Theorem~\ref{thm:visib_properties}, we need some general structural results regarding relative symplectic cohomology.
 
First, we show that $SH_M(K;\Lambda)$ is actually a unital $\Lambda$-algebra,  and the units are functorial under  restriction maps. While this is a totally expected structure, setting up the unit in the context of relative symplectic cohomology is not that straightforward and requires  technical care. Our solution is to define the unit via {\it raised symplectic cohomology},  a modification of the original construction which allows us to use strictly positive Hamiltonians as perturbations.

Second, let $L\subset M$ be a tautologically unobstructed Lagrangian submanifold satisfying technical conditions related to gradings and signs, and $K\subset M$ be a compact subset. We introduce relative Lagrangian (self-)Floer cohomology, $HF^*_M(L,K,\Lambda)$ which is also a unital $\Lambda$-algebra. We also introduce the closed-open map
$$
\cC\cO\co SH^*_M(K,\Lambda)\to HF^*(L,K,\Lambda).
$$
We shall prove that it is a map of $\Lambda$-vector spaces respecting units, but it is easy to modify our arguments and show that it is a ring map.

\begin{rem}
In this paper, we do not discuss most of the expected properties of the ring structures, as we have no use for them. Namely, we do not prove that our ring structures do not depend on the choices, that they are associative, or in the closed string case that they are commutative. We do not expect any of these proofs to require any new ideas. On the other hand the algebraic formalism that we use is a cumbersome one to write these proofs rigorously. The second author is currently working with Mohammed Abouzaid and Yoel Groman on a more natural (albeit quite a bit more technical) formalism to construct algebraic structures on relative invariants (as of now this excludes units). It seems like a better idea to prove general properties in that framework and then show that the product structure we produce here agrees with the one which is constructed in this forthcoming work.
\end{rem}

Relative Lagrangian (self-)Floer cohomology also satisfies the Mayer-Vietoris property. The relevant statement for this paper is the last item in Proposition \ref{prop:lag_analog_thesis} below, but one can construct a Mayer-Vietoris sequence with the same level of generality as in \cite{Va18}. With this in mind, let us mention another consequence of $SH$-fullness.
\begin{cor}
Let $L\subset M$ be a tautologically unobstructed oriented Lagrangian submanifold with a relative Pin structure, and $K$ be $SH$-full. Let $D$ be any compact domain containing $K$, then the restriction map $HF(L,\Lambda)\to HF(L,D,\Lambda)$ is an isomorphism.
\end{cor}

The proof is similar to the proof of Theorem \ref{thm:visib_properties} Item (6) given at the end of Section \ref{sslag}.

\subsection*{Structure of the paper}
 In Section 2, we review the construction of relative symplectic cohomology and introduce its open string version. We reduce the proofs of properties (4) and (6) of Theorem \ref{thm:visib_properties} to the existence of units with special properties. In Section 3, we discuss some chain level algebra which will be used to put algebra structures and define units on $H(\widehat{tel}(\mathcal{C}))$, where $\mathcal{C}$ is a $1$-ray. In Section 4, we apply the contact Fukaya trick to prove Proposition \ref{prop:index_bd_collar_invt}. In Section 5, we define the unit in relative symplectic cohomology and show that it is preserved under restriction maps. In Section 6, we define the unit in relative Lagrangian cohomology and construct closed-open maps. Then, we show that units are preserved under restriction maps and closed-open string maps.

In Appendix A, we give a proof of Proposition \ref{prop:giroux_skeleton}.

\subsection*{Convention} In what follows chain complex means what is conventionally called cochain complex - all of our differentials in the $\mathbb{Z}$-graded context increase degree, yet we still call them chain complex for convenience.

\subsection*{Acknowledgements}
We have drawn inspiration from the work of Mark~McLean. We thank him for his comprehensive papers. We also thank Paul Seidel for giving us comments on an earlier version of the Introduction to the paper. Finally, we thank KIAS for giving us a chance to collaborate at their institution for a week at the end of 2019 Summer.

DT was partially supported by the Simons
Foundation under grant \#385573, Simons Collaboration on Homological Mirror Symmetry.

\section{Relative Floer theoretic invariants}
\label{sec:rel_sh}
 In this section we go over the construction of relative symplectic cohomology from ~\cite{Va18}, and the  Lagrangian Floer cohomology version which will also be used in this paper. We prove the missing Items~(4) and~(6) from Theorem~\ref{thm:visib_properties}, modulo important technical propositions which are later proved in Sections~\ref{sec:prod_unit} and~\ref{sec:lag}.

\subsection{Reminder on relative symplectic cohomology}\label{ssremindersh}
Let $M$ be a closed symplectic manifold. Let us assume that $c_1(M)=0$ for simplicity, and also fix a trivialization of the canonical bundle of $M$. Note that this choice does not play a role in the grading of contractible orbits. We refer the reader to \cite{Va18} for the construction without the Calabi-Yau assumption, which requires virtual techniques.

Let $K\subset M$ be a compact subset. We call the following datum an acceleration datum for $K$:
\begin{itemize}
\item $H_1\leq H_2\leq\ldots$ a monotone sequence of non-degenerate one-periodic Hamiltonians $H_i\co M\times S^1\to \bR$ cofinal among functions satisfying $H\mid_{S^1\times K}<0$. The latter condition is equivalent to
$$
H_i(x,t)\xrightarrow[i\to+\infty]{}\begin{cases}
0,& x\in K,\\
+\infty,& x\notin K.
\end{cases}
$$
\item A monotone homotopy of Hamiltonians $H_{i,i+1}:[i,i+1]\times M\times S^1\to\mathbb{R}$, for all $i$, which is equal to $H_i$ and $H_{i+1}$ in a neighborhood of the corresponding end points.
\end{itemize}

One can combine an acceleration datum into a single family of time-dependent Hamiltonians $H_s\co M\times S^1\to \bR$, $s\in \bR_{\geq 1}$. We also fix a Morse function on $[0,1]$ with critical values at the end points once and for all, which turns a $[0,1]$-family of Hamiltonians to a $(-\infty,\infty)$-family which is then used to write down the Floer equations.

%

Given an acceleration datum and a choice of a generic time dependent almost complex structure $J$,
Hamiltonian Floer theory provides a $1$-ray of  Floer chain complexes over $\Lambda_{\geq 0}$, called a {\it Floer 1-ray}: \begin{align*}
\mathcal{C}(H_r,J):= CF^*(H_1)\to CF^*(H_2)\to\ldots 
\end{align*}
Each $CF^*(H_i)$ is the Floer complex of $H_i$ over $\Lambda_{\ge 0}$, with the usual Floer differential. 
The horizontal arrows are Floer continuation maps defined using the monotone homotopies appearing in the acceleration datum. Recall that a cylinder $u$ contributing to a Floer differential or a continuation map, does so with Novikov weight 
$$
T^{E_{top}(u)}
$$
where
\begin{equation}
\label{eq:E_top}
E_{top}(u)=\int_{S^1}\gamma_{out}^* H_{out}\,dt-\int_{S^1}\gamma_{in}^* H_{in}\,dt+\omega(u),
\end{equation}
$\gamma_{out}$, $\gamma_{in}$ are the asymptotic orbits of $u$, and $H_{out}$, $H_{in}$ are the Hamiltonians at the correspondind ends. (For Floer differentials, $H_{out}=H_{in}=H_i$ and for continuation maps, $H_{out}=H_{i+1}$, $H_{in}=H_i$ for some $i$.)

One defines the $\Lambda_{\geq 0}$-chain complex $$tel(\mathcal{C}(H_s,J)) \text{ and } \widehat{tel}(\mathcal{C}(H_s,J)),$$
as in \cite{Va18}. see also Section~\ref{sec:alg} below. We stress that we always take the degreewise completion.


For more details on the following proposition (and a proof) see Proposition 3.3.2 from \cite{Va18}.
\begin{prop}
	For two different choices of acceleration data for $K$ and almost complex structures, $H_s,J$ and $H_s',J'$,
there is a canonical $\Lambda_{\geq 0}$-module map from $H^*(\widehat{tel}(\mathcal{C}(H_s,J)))$ to $H^*(\widehat{tel}(\mathcal{C}(H'_s,J')))$ defined using monotone continuation maps. Moreover, these maps are isomorphisms.
\qed
\end{prop}

Hence, we define \begin{align*} 
SH_M^*(K):=H^*(\widehat{tel}(\mathcal{C}(H_s,J))).
\end{align*}

In what follows we will only be interested in the torsion-free part of $SH_M^*(K)$: \begin{align}\label{eqdefsh}SH_M^*(K,\Lambda):=SH_M^*(K)\otimes\Lambda.
\end{align}

\begin{rem}
	Typically in Floer theory, one requires $J$ to be compatible with $\omega$. In fact, the theory works the same under the following weaker assumption: $J$ is tamed by $\omega$, and compatible in a neighbourhood of all 1-periodic orbits of all of the $H_i$. This greater flexibility will be later used in setting up the contact Fukaya trick.
\end{rem}

\begin{rem}
	
	When $M$ is  Calabi-Yau,
	Hamiltonian Floer theory on $M$ can be set up using classical techniques, see Lecture 3 of \cite{sal}, using generic choices of almost complex structures and Hamiltonians, which is what we are doing here. Note that in the Calabi-Yau case the genericity does not fail even if we use higher-dimensional parametric families of Floer equations. Indeed, the loci of Chern number~zero  $J$-holomorphic spheres are always codimension $4$ in $M$ times any parameter space. Therefore they generically do not interact with rigid or one-dimensional moduli spaces of Floer solutions. Note that here the key property is that there are no negative Chern number spheres, and therefore it is not possible to converge to configurations that involve Floer solutions that belong to higher dimensional families. The details of how transversality can be achieved by perturbing almost complex structures (for fixed Hamiltonian data) was worked out in Appendix B of \cite{mclean}, which we will also be using here.
	
\end{rem}

\begin{rem}\label{remanalytic}Let us now explain a slightly different way of obtaining $SH_M^*(K)\otimes\Lambda$ using a more analytic language. Note that $\Lambda$ is a non-archimedean valued field. We define $CF^*(H_i,\Lambda)$ to be the (non-archimedean) normed/valued\footnote{It is more convenient to talk about valuations for us. For translation to the more familiar language, as it is used in functional analysis for the archimeden case: note that norm is given by $e^{-val}.$} $\Lambda$-vector space freely generated by the $1$-periodic orbits which all have valuation $0$ (and the valuation of a $\Lambda$-linear combination of these basis elements is defined by taking the infimum of the valuations of the summands). Similarly, we define $tel(\mathcal{C}(H_s,J,\Lambda))$ as a valued $\Lambda$-vector space, which is notably infinite dimensional. Note that as $\Lambda$-vector spaces $tel(\mathcal{C}(H_s,J,\Lambda))$ and $tel(\mathcal{C}(H_s,J))\otimes \Lambda$ are canonically isomorphic. Now we take the completion of $tel(\mathcal{C}(H_s,J,\Lambda))$ and obtain a $\Lambda$-Banach space: $\widehat{tel}(\mathcal{C}(H_s,J,\Lambda))$. The homology of this complex is canonically isomorphic to $SH_M^*(K)\otimes\Lambda$.\qed
\end{rem}

Recall from Proposition 3.3.3 of  \cite{Va18}
	that if $K\subset K'\subset M$, there is a canonical restriction map of $\Lambda_{\geq 0}$-modules
$$
SH_M^*(K')\to SH_M^*(K).
$$

\subsection{Units and visibility}

We move on to new statements, no longer from \cite{Va18}. The next proposition  will be proved in Section~\ref{sec:prod_unit}.

\begin{prop}
	\label{prop:unit_sh}
		For every compact set $K\subset M$, there is a distinguished element $1_K\in SH_M(K,\Lambda)$, called the unit, with the following properties.
\begin{itemize}	
\item  $SH_M(K,\Lambda)=0$ if and only if $1_K=0$. 
\item Restriction maps send units to units.
\end{itemize}
\end{prop}

Although this proposition does not mention the existence of the product structure, we do actually prove its existence and unitality in Section~\ref{sec:prod_unit}.

\begin{proof}[Proof of Theorem~\ref{thm:visib_properties}, Item~(4)]
	We use Proposition~\ref{prop:unit_sh}.
	Suppose $K'\subset K$,
	then the restriction map $SH^*_M(K,\Lambda)\to SH^*_M(K',\Lambda)$ sends $1_K$ to $1_{K'}$. Suppose that $SH^*_M(K,\Lambda)=0$, then $1_{K}=0$, so $1_{K'}=0$. Consequently, $SH^*_M(K',\Lambda)=0$.
\end{proof}	

\subsection{Relative Lagrangian Floer theory}\label{sslag}
Let  $L\subset M$ be an oriented Lagrangian with a relative Pin structure. Assume that there is a compatible (time-independent) almost complex structure $J_L$ such that $L$ does not bound any $J_L$-holomorphic discs. We want to define the relative Lagrangian Floer homology of $(L,J_L)$ for any compact subset $K\in M$: $$HF(L,J_L,K).$$ This will be a $\mathbb{Z}/2\mathbb{Z}$-graded $\Lambda$-vector space. The construction is very similar to the definition of relative symplectic cohomology, so we will be brief. We are using the results of \cite{Sei15} and \cite{mclean} here.

For a Hamiltonian $H:[0,1]\times M\to \mathbb{R}$ such that $\phi_H^1(L)$ is transverse to $L$, and a generic $[0,1]$-dependent almost complex structure $J:=\{J_t\}_{t\in [0,1]}$ with $J_t(x)=J_L(x)$ for every $x\in L$ and $t\in [0,1]$, we obtain $$CF^*(L,,J_L,H,J,\Lambda_{\geq 0}):=\left(\bigoplus_{\text{H-chords}}\Lambda_{\geq 0},d_{Fl}\right),$$ which is a chain complex over $\Lambda_{\geq 0}$ generated by the $1$-chords of $H$, and the differential counts Floer solutions $u: \mathbb{R}\times [0,1]\to M$ with boundary mapping to $L$ with weights \begin{equation}
\label{eq:E_top_lag}
E_{top}(u)=\int_{[0,1]}\gamma_{out}^* H_{out}\,dt-\int_{[0,1]}\gamma_{in}^* H_{in}\,dt+\omega(u).
\end{equation}

For a monotone homotopy $H:[0,1]\times [0,1]\times M\to \mathbb{R}$ with $H|_0=H_0$ and $H|_1=H_1$, and a generic $[0,1]_t\times [0,1]_s$-dependent almost complex structure $J:\{J_{s,t}\}_{s,t\in [0,1]}$ with $J_{s,t}(x)=J_L(x)$ for every $x\in L$ and $s,t\in [0,1]$ and $J_{0,t}$ and $J_{1,t}$ generic as in the previous paragraph, we obtain a chain map 
$$CF^*(L,J_L,H_0,J_{0,t},\Lambda_{\geq 0})\to CF^*(L,J_L,H_1,J_{1,t},\Lambda_{\geq 0}).$$ A generic homotopy rel endpoints of such data gives rise to a chain homotopy as usual.

We can then define the $1$-ray: 
$$\mathcal{C}:=CF^*(L,H_1,J,\Lambda_{\geq 0})\to CF^*(L,H_2,J,\Lambda_{\geq 0})\to \ldots\to ,$$ and define $$CF^*(L,H_s,J):=\widehat{tel}(\mathcal{C}).$$

The same arguments with the closed string case as in Section \ref{ssremindersh} show that $H(CF^*(L,H_s,J))$ is invariant under the choices of $H_s$ and $J$. Hence, we defined our $HF^*(L,J_L,K).$ As before, we want to throw away the torsion part of this, and define $$HF^*(L,J_L,K,\Lambda):=HF^*(L,J_L,K)\otimes\Lambda.$$

We will also drop the $J_L$ from the notation from now on, and declare that $L$ being a tautologically unobstructed Lagrangian means that it has a specified $J_L$ implicit in writing $L$.

\begin{prop}
	\label{prop:lag_analog_thesis}\begin{enumerate}
\item	There are canonical restriction maps for $K\subset K'$:
$$HF^*(L,K')\to HF^*(L,K).
$$

\item $HF^*(L,M,\Lambda)$ is isomorphic to Lagrangian Floer homology of $L$, which is isomorphic to $H^*(L,\Lambda)$ as a $\Lambda$-vector space.
	
\item	 If $L$ lies in the complement of $K$ in $M$, then $HF^*(L,K)=0$.
	
\item Let $K_1$ and $K_2$ be compact domains with equal boundaries such that $K_1\cup K_2=M$, then we have a Mayer-Vietoris sequence:
$$
 \begin{tikzcd}
HF^*(L)\ar[r]&HF^*(L,K_1)\oplus HF^*(L,K_2)\ar[dl]\\HF^*(L,K_1\cap K_2)\ar[u]
\end{tikzcd}
$$ 

\end{enumerate}
\end{prop}

\begin{proof}
The first two statements are straightforward. For the third statement, note that we can choose an acceleration data for $K$, so that all the $1$-chords of $L$ lie outside $K$, and moreover the value of the Hamiltonian $H_n$ at the chords of $H_n$ is approximately $n$. Now by the ``adiabatic'' argument in  \cite{VaThesis}, we obtain a uniform lower bound on the topological energies of all possible continuation maps for slowed down acceleration data, which proves that none of the generators survive in the completion. For the last one, note that the proof from \cite{Va18} applies verbatim here as one simply replace the $1$-periodic orbits in Lemma 4.1.1 from \cite{Va18}  with $1$-chords on $L$, and the acceleration data that is constructed there would also satisfy this modified requirement. 
\end{proof}

\begin{prop}
	\label{prop:unit_hf}
There is a distinguished  element $1_{K,L}\in HF^*(L,K,\Lambda)$, called the unit, with the following properties.

\begin{itemize}
	\item $HF^*(L,K,\Lambda)=0$ if and only if $1_{K,L}=0$. 
	
	\item 
	Restriction maps send units to units.
\end{itemize}
	
\end{prop}

We need one more piece of information, which uses the existence relative closed-open string maps with good properties.

\begin{prop}
	\label{prop:zero_unit_hf_sh}
For any $K,L$ as above, if $1_K=0\in SH^*_M(K,\Lambda)$, then $1_{K,L}=0\in HF^*(L,K,\Lambda)$.
\end{prop}

The proofs of Propositions ~\ref{prop:unit_hf} and~\ref{prop:zero_unit_hf_sh} are given in Section~\ref{sec:lag}.

\begin{proof}[Proof of Theorem~\ref{thm:visib_properties}, Item~(6)]
Suppose $L\subset M$ be such a Lagrangian submanifold, in particular $L\subset int(K)$.

By definition, we have that $SH^*_M(K,\Lambda)=0$. First, we claim that $HF^*(L,K,\Lambda)=0$. This follows by the unitality trick: we have that $1_K=0$ by Proposition~\ref{prop:unit_sh}, hence $1_{K,L}=0$ by Proposition~\ref{prop:zero_unit_hf_sh}, so $HF^*(L,K,\Lambda)=0$ by Proposition~\ref{prop:unit_hf}.

Let $N$ be the compact domain that is the closure of $M\setminus K$. By Proposition~\ref{prop:lag_analog_thesis}, $HF^*(L,N,\Lambda)=0$.

Note that since units are preserved under restriction maps, we also have that $HF^*(L,K\cap N,\Lambda)=0$.
Now the Lagrangian Mayer-Vietoris sequence from Proposition~\ref{prop:lag_analog_thesis} implies that $HF^*(L)=0$, which is a contradiction.
\end{proof}

\section{Chain-level algebra}
\label{sec:alg}
This section sets up some algebraic preliminaries used later. The first subsection reminds the basic notions from \cite{Va18}; the rest contains more technical material which will be necessary in Sections ~\ref{sec:prod_unit} and \ref{sec:lag}.

\subsection{Rays, completion, telescope}
This subsection reminds the algebraic setup from \cite{Va18}; we assume the reader is familiar with this reference.

Let $Ch_{\Lambda_{\geq 0}}$ be the category of $\mathbb{Z}$- or $\mathbb{Z}/2\mathbb{Z}$- graded chain complexes over the Novikov ring $\Lambda_{\geq 0}$.
A 1-ray $\cC$ is the following diagram in $Ch_{\Lambda_{\geq 0}}$, infinite to the right:
$$
\cC=C_1\xrightarrow{c_1} C_2\xrightarrow{c_2} C_3\xrightarrow{c_3} \ldots 
$$
Here each $C_i$ is a chain complex over $\Lambda_{\geq 0}$, and each $c_i$ is a chain map.
Let $1-ray$ be the category of $1$-rays with underlying modules assumed to be free. Morphisms in this category are given by maps of $1$-rays, and composition of morphisms is defined by composing the squares in the finite direction. 

 In $1-ray$ we have a notion of two morphisms being {\it equivalent}, defined by the existence of a homotopy of maps of $1$-rays. 



The telescope construction provides a functor $$tel: 1-ray\to Ch_{\Lambda_{\geq 0}}.$$ We also have the degree-wise completion functor $$\widehat{\cdot}:Ch_{\Lambda_{\geq 0}}\to Ch_{\Lambda_{\geq 0}}.$$ Composing the telescope and the completion functor, we obtain the completed telescope functor $$\widehat{tel}: 1-ray\to Ch_{\Lambda_{\geq 0}}.$$

Note that $tel$ sends equivalent morphisms to homotopy equivalent chain maps. Moreover, weak equivalences are sent to quasi-isomorphisms. The same statements hold true after completion as well, noting that quasi-isomorphisms stay quasi-isomorphisms after completion whenever the underlying modules are free.

\subsection{Filtered direct limits and strictification}
Another way to express the data of a $1$-ray is the following. Let $N$ be the category with objects positive integers, and precisely $1$ morphism from $n$ to $m$, whenever $n\leq m$, and no other morphisms. Then a functor $N\to Ch_{\Lambda_{\geq 0}}$ is precisely the same data as a $1$-ray. 

Let us denote the category of $\mathbb{Z}$- or $\mathbb{Z}/2\mathbb{Z}$- graded modules over $\Lambda_{\geq 0}$ by $Mod_{\Lambda_{\geq 0}}$. Yet another way to describe a $1$-ray is as a  functor $N\to Mod_{\Lambda_{\geq 0}}$ (forgetting the differentials), and a natural transformation of this functor to its composition with the shift functor (given by the differentials).

With this in mind, given a $1$-ray $\mathcal{C}=C_1\to C_2\to C_3\to\ldots$, we can also define a chain complex $$\varinjlim{(\mathcal{C})},$$ which is (for definiteness) obtained by applying the standard construction of filtered direct limits of modules to the corresponding functor $N\to Mod_{\Lambda_{\geq 0}}$ and to the natural transformation given by the differentials.

Let us call a map of $1$-rays $\mathcal{C}\to \mathcal{C}'$ \textbf{strict} if for all $i\geq 1$, $C_i\to C'_{i+1}[1]$ (i.e. the homotopies) are identically zero. Let us denote the corresponding subcategory of $1-ray$ by $st-1-ray$. Note that $\varinjlim$ defines a functor $$\varinjlim: st-1-ray\to Ch_{\Lambda_{\geq 0}}.$$

We define the $1$-ray $$F(\mathcal{C})=F^1(tel(\mathcal{C}))\to F^2(tel(\mathcal{C}))\to F^3(tel(\mathcal{C}))\to\ldots,$$ where $$F^n(tel(\mathcal{C}))=\left(\bigoplus_{i\in [1,n-1]}C_i[1]\oplus C_i\right)\oplus C_n$$ with the differential depicted below (for $n=3$, the general form is clear) \begin{align}\label{teles}
\xymatrix{
C_1\ar@{>}@(ul,ur)^{d }  &C_2\ar@{>}@(ul,ur)^{d} &C_3\ar@{>}@(ul,ur)^{d}\\
C_1[1]\ar@{>}@(dl,dr)_{-d} \ar[u]^{\text{id}}\ar[ur]^{f_1} &C_2[1]\ar@{>}@(dl,dr)_{-d} \ar[u]^{\text{id}}\ar[ur]^{f_2}&}
\end{align} and the maps are the canonical inclusion maps. Note that $\varinjlim{(F(\mathcal{C}))}$ is canonically isomorphic to $tel(\mathcal{C})$. 

Then by the discussion of functioriality of cones and telescopes as in \cite{Va18}, it is clear that we can extend $F$ to a functor $$F:1-ray\to st-1-ray.$$ We call this the \textbf{strictification} functor.

Let us also note that there exists canonical commutative diagrams
\begin{align}\label{c2estrictification}
\xymatrix{
F^n(tel(\mathcal{C})) \ar[d]\ar[r]  &F^{n+1}(tel(\mathcal{C}))\ar[d] \\ C_n \ar[r] &C_{n+1}},
\end{align} where the vertical arrows are quasi-isomorphisms given by the direct sum of maps $C_i\to C_n$ multiplied by $(-1)^{i}$, $i\in [1,n]$ and the zero maps $C_i[1]\to C_n$, $i\in [1,n-1]$.

These diagrams give a morphism\footnote{This is a cofibrant replacement in the Reedy model structure of $st-1-ray$ obtained using the projective model structure on $Ch_{\Lambda_{\geq 0}}$.} in $st-1-ray$, which induce quasi-isomorphisms after applying the $\varinjlim$ functor.

\subsection{Tensor product}
\label{subsec:tensor}
The remainder of this section is no longer taken from \cite{Va18}. It will be used to set up the product and units on relative symplectic cohomology\footnote{During the revision, we noticed the existence of the relevant paper \cite{greenlees}. The explicit map in Lemma 0.1 would have simplified (not by much) some of the discussion below (compare with our Lemma \ref{lem:tel_tensor}). It is not just tedious, as the authors point out, but also tricky to check that this map is a chain map. We also take this opportunity to refer the reader to the elementary viewpoint on the telescope construction given in the paragraph below this lemma. This version of the telescope differs from ours by the sign of the diagonal arrows in \eqref{teles}. The two versions are isomorphic using the direct sum of the identity maps $C_i\to C_i$ and $C_i[1]\to C_i[1]$ multiplied  by $(-1)^i$. This explains the sign that we needed to introduce in  \eqref{c2estrictification}.}
  
Let us define the tensor product $\mathcal{C}\otimes \mathcal{C}'$ of two one-rays $\mathcal{C}$ and $\mathcal{C}'$ as their slice-wise tensor product.That is, if
$
\cC=C_1\to C_2\to\ldots
$
and 
$\cC'=C_1'\to C_2'\to\ldots$, then

$$\cC\otimes\cC'=C_1\otimes C_1'\to C_2\otimes C_2'\to\ldots $$
with the obvious structure maps.  Note that the differential of the tensor product of two chain complexes involves the Koszul sign as usual. We shall use the following\footnote{Looking at the proof of Lemma 0.1 from \cite{greenlees} the reader might question why we had to introduce $\mathcal{E}$ instead of using $\varinjlim{(\mathcal{C})}\otimes \varinjlim{(\mathcal{C}')}$. The reason is that we want the right pointing arrows in the zig-zag to stay quasi-isomorphisms after completion and this was more apparent to us with $\mathcal{E}$. Of course, if we just used Lemma 0.1, we would not have to think about this technicality, as the quasi-isomorphism there goes between free modules.}

\begin{lem}
	\label{lem:tel_tensor}
Let $\mathcal{E}:=\varinjlim{(F^1(tel(\mathcal{C}))\otimes F^1(tel(\mathcal{C}')\to\ldots \to F^n(tel(\mathcal{C}))\otimes F^n(tel(\mathcal{C}'))\to \ldots)}$.
There exists a canonical zig-zag of quasi-isomorphisms $$tel(\mathcal{C}\otimes \mathcal{C}')\to \lim{(\cC\otimes\cC')}\leftarrow \mathcal{E}\to tel(\mathcal{C})\otimes tel(\mathcal{C}').$$
Moreover, the maps $tel(\mathcal{C}\otimes \mathcal{C}')\to \varinjlim{(\cC\otimes\cC')}$ and $\mathcal{E}\to tel(\mathcal{C})\otimes tel(\mathcal{C}')$ stay quasi-isomorphisms after completion.
\end{lem}

\begin{proof}
Note that we have canonical quasi-isomorphisms 
$$F^n(tel(\mathcal{C}\otimes \mathcal{C}'))\to C_n\otimes C_n'\leftarrow F^n(tel(\mathcal{C}))\otimes F^n(tel(\mathcal{C}')),$$ which are compatible for different $n$'s. Taking the direct limits of these maps we obtain a zig-zag of quasi-isomorphisms:
$$tel(\mathcal{C}\otimes \mathcal{C}')\to \varinjlim{(\cC\otimes\cC')}\leftarrow \mathcal{E}.$$

Clearly, $$F^1(tel(\mathcal{C}))\otimes F^1(tel(\mathcal{C}')\to\ldots \to F^n(tel(\mathcal{C}))\otimes F^n(tel(\mathcal{C}')\to \ldots$$ is cofinal in the diagram 
\begin{align*}
\xymatrix{
F^1(tel(\mathcal{C}))\otimes F^1(tel(\mathcal{C}') \ar[d]\ar[r]  &F^1(tel(\mathcal{C}))\otimes F^2(tel(\mathcal{C}')\ar[d]\ar[r]& F^1(tel(\mathcal{C}))\otimes F^3(tel(\mathcal{C}')\ar[r]\ar[d]& \\ F^2(tel(\mathcal{C}))\otimes F^1(tel(\mathcal{C}')\ar[r]\ar[d] &F^2(tel(\mathcal{C}))\otimes F^2(tel(\mathcal{C}')\ar[r]\ar[d]&& \\ F^3(tel(\mathcal{C}))\otimes F^1(tel(\mathcal{C}')\ar[r]\ar[d]&&\ldots &\\&&&&}
\end{align*}

Moreover, the filtered direct limit of the diagram above is canonically isomorphic to $tel(\mathcal{C})\otimes tel(\mathcal{C}')$. Therefore we obtain a quasi-isomorphism $$\mathcal{E}\to tel(\mathcal{C})\otimes tel(\mathcal{C}'),$$which finishes the proof of the first statement.

That $tel(\mathcal{C}\otimes \mathcal{C}')\to \varinjlim{(\cC\otimes\cC')}$ stays a quasi-isomorphism after completion follows from Lemma 2.3.7 of \cite{Va18}. More straightforwardly, the same is true for $\mathcal{E}\to tel(\mathcal{C})\otimes tel(\mathcal{C}')$ because the involved modules are free.

	\end{proof}

Assume that we are given a morphism $\mathcal{C}\otimes \mathcal{C}'\to \mathcal{D}$ in $1-ray$. We can  turn this into a map \begin{align}\label{homologyproduct}
H^*(\widehat{tel}(\mathcal{C}))\otimes H^*(\widehat{tel}(\mathcal{C}'))\to H^*(\widehat{tel}(\mathcal{D}))
\end{align} by composing the following natural maps. 

\begin{enumerate}
\item $H^*(\widehat{tel}(\mathcal{C}))\otimes H^*(\widehat{tel}(\mathcal{C}'))\to H^*(\widehat{tel}(\mathcal{C})\otimes \widehat{tel}(\mathcal{C}'))$;
\item $H^*(\widehat{tel}(\mathcal{C})\otimes \widehat{tel}(\mathcal{C}'))\to H^*(\widehat{tel(\mathcal{C})\otimes tel(\mathcal{C}')})$;
\item $H^*(\widehat{tel(\mathcal{C})\otimes tel(\mathcal{C}')})\to H^*(\widehat{tel}(\mathcal{C}\otimes \mathcal{C}'))$;
\item $H^*(\widehat{tel}(\mathcal{C}\otimes \mathcal{C}'))\to H^*(\widehat{tel}(\mathcal{D}))$.
\end{enumerate}

The third map comes from Lemma~\ref{lem:tel_tensor}, and the other maps are obvious.

\begin{lem}\label{lemmafunctorial}
Assume that we have a morphism $\mathcal{C}\to\tilde{\mathcal{C}}$ and an object $\mathcal{C}'$ in $1-ray$, then we get a morphism $\mathcal{C}\otimes \mathcal{C}'\to\tilde{\mathcal{C}}\otimes \mathcal{C}'$. In turn, if we have a morphism $\tilde{\mathcal{C}}\otimes \mathcal{C}'\to\mathcal{D}$, by composition we can obtain $\mathcal{C}\otimes \mathcal{C}'\to\mathcal{D}$.

Further assume that $\mathcal{C}\to\tilde{\mathcal{C}}$ is strict\footnote{This is the main place where Lemma 0.1 of \cite{greenlees} would have helped us by removing this restriction. Note that the map in Lemma 0.1 is also not strictly (as opposed to up to homotopy) natural for maps of $1$-rays in both $X$ and $Y$ slots in their notation (the asymmetry of the formula gives a hint). To make this statement we use the functoriality of the telescope construction and also the map from the first paragraph of Lemma \ref{lemmafunctorial}. The map of Lemma 0.1  is only natural in $Y$, as another tricky computation shows, but this would be enough for our purposes.}. Then, the diagram below commutes:
$$
\begin{tikzcd}
H^*(\widehat{tel}(\mathcal{C}))\otimes H^*(\widehat{tel}(\mathcal{C}'))
\ar[d]
\ar[r]
&
H^*(\widehat{tel}(\mathcal{D}))
\\
H^*(\widehat{tel}(\tilde{\mathcal{C})})\otimes H^*(\widehat{tel}(\mathcal{C}'))
\ar[ur]
\end{tikzcd}
$$
\end{lem}

\begin{proof}A direct computation shows that the diagram below is a $1$-cube.
$$
\begin{tikzcd}
C_i\otimes C_i'
\ar[d, " g_i\otimes id"']
\ar[r, "f_i\otimes f_i'"]
\ar[dr,"h_i\otimes f_i'"]
&
C_{i+1}\otimes C_{i+1}'
\ar[d, " g_{i+1}\otimes id"]
\\
\tilde{C}_i\otimes C_i'  
\ar[r, "\tilde{f}_i\otimes f_i'"']
&
\tilde{C}_{i+1}\otimes C_{i+1}'
\end{tikzcd}
$$

%
%
%

This finishes the first part. The second part follows from showing the naturality of maps (1)-(3). The strictness assumption helps with (3) as $\varinjlim$ is only functorial for strict morphisms of $1$-rays. Here we also use that the diagrams (\ref{c2estrictification}) are compatible with strict morphisms.
\end{proof}

\subsection{Units}
\label{subsec:algebra_units}

Let $U$ be the $1$-ray 
$$\Lambda_{\geq 0}\to\Lambda_{\geq 0}\to\ldots ,$$ where all complexes $\Lambda_{\ge 0}$ are in degree $0$ and have zero differential, and the morphism maps are all the identity. Note that, as suggested by the notation, we are remembering the unit element of $\Lambda_{\ge 0}$ here. Both $H^*(tel(U))$ and $H^*(\widehat{tel}(U))$ are canonically isomorphic to $\Lambda_{\geq 0}$ in degree $0$, and zero otherwise.

\begin{defin}\label{defrealization}Let 
$$f:\mathcal{C}\to \mathcal{D},\ U_{\mathcal{C}'}:U\to \mathcal{C}',\text{ and }p:\mathcal{C}'\otimes \mathcal{C}\to \mathcal{D}$$ be morphisms in $1-ray$. If the composition 
$$\mathcal{C}= U \otimes \mathcal{C}\xrightarrow{} \mathcal{C}'\otimes \mathcal{C}\xrightarrow{p} \mathcal{D}$$ is equivalent to $f$, we call $U_{\mathcal{C}'}$ a realization of $f$ via $p$.
\end{defin}

Let us analyze the situation more concretely. A map of $1$-rays $U_{\mathcal{C}'}:U\to \mathcal{C}'$ is equivalent to the following data. \begin{itemize}
\item a closed element $u_i\in C_i^{0'}$ for each $i=1,2,\ldots$
\item an element $p_{i+1}\in C_{i+1}^{-1'}$ such that the image of $u_i$ in $C_{i+1}'$ is equal to $u_{i+1}+dp_{i+1}$.
\end{itemize}

Let $u_{\mathcal{C}'}$ be the image of $1$ under the map $\Lambda_{\geq 0}=H^*(\widehat{tel}(U))\to H(\widehat{tel}(\mathcal{C}'))$. The following lemma is easy.

\begin{lem}
Let $x\in C_1'$ be any element cohomologous to $u_1$, then the homology class of its image under the canonical chain map $C_1'\to \widehat{tel}(\mathcal{C}')$ is equal to $u_{\mathcal{C}'}$.
\end{lem}

Now let us take a map of $1$-rays $U_{\mathcal{C}'}^*:U\to \mathcal{C}'$ equivalent to $U_{\mathcal{C}'}$. The data of $U_{\mathcal{C}'}^*$ is equivalent to elements $u_i^*,p_{i+1}^*$ as above. The homotopy between $U_{\mathcal{C}'}$ and $U_{\mathcal{C}'}^*$ is equivalent to the data of 

\begin{itemize}
\item an element $h_i\in C_i^{-1'}$ such that $u_i-u_i^*=dh_i$ for each $i=1,2,\ldots$
\item an element $q_{i+1}\in C_{i+1}^{-2'}$ such that the image of $h_i$ in $C_{i+1}'$ is equal to $h_{i+1}+dq_{i+1}$.
\end{itemize}

Finally we want to analyze the maps 
$$\mathcal{C}= U\otimes \mathcal{C}\xrightarrow{} \mathcal{C}'\otimes \mathcal{C}$$ for $U_{\mathcal{C}'}$ and $U_{\mathcal{C}'}^*$. More precisely, we want to show that these two maps are equivalent.

$$
\begin{tikzcd}
C_i
\ar[d, " u_i\otimes", bend left]
\ar[d, " u_i^*\otimes"', bend right]
\ar[r, "f_i"'']
&
C_{i+1}
\ar[d, "u_{i+1}\otimes ", bend left]
\ar[d, " u_{i+1}^*\otimes"', bend right]
\\
C_i'\otimes C_i  
\ar[r, "f_i'\otimes f_i"']
&
C_{i+1}'\otimes C_{i+1}
\end{tikzcd}
$$

We define $C_i\to C_i\otimes C_i'[1] $ by $c\mapsto  h_i\otimes c$. Let us check that this indeed is a homotopy: $$ u_i\otimes c- u_i^*\otimes c=dh_i \otimes c = d(h_i\otimes c)+h_i\otimes dc.$$

We then recall that the two maps $C_i\to C_{i+1}\otimes C_{i'+1}[1] $ (corresponding to the two squares on the sides of the $2$-slit) are by definition $c\mapsto p_i\otimes f_i(c) $ and $c\mapsto  p_i^*\otimes f_i(c)$. 


Finally we define $C_i\to C_i\otimes C_i'[2] $ by $c\mapsto q_{i+1}\otimes f_i(c)$. Again a similar computation shows that this completes the diagram to a $3$-slit. Note that there are two sign changes in the computation, which result in the correct equation.


Hence, we proved the following.

\begin{lem}\label{lemstrictunit}
Assume that $U_{\mathcal{C}'}$ is a realization of $f$ via $p$ as above as in Definition \ref{defrealization}. Then any morphism that is equivalent to $U_{\mathcal{C}'}$ is also a realization of $f$ via $p$. Moreover, it gives rise to the same element  $u_{\cC'}\in H(\widehat{tel}(\mathcal{C}'))$.
\end{lem}

The next lemma is the main result of this section.

\begin{lem} 
	\label{lem:realis_h}
	Assume that $U_{\mathcal{C}'}$ is a realization of $f$ via $p$. Let $u_{\mathcal{C}'}$ be the image of $1$ under the map $\Lambda=H^*(\widehat{tel}(U))\to H(\widehat{tel}(\mathcal{C}'))$ as before. Then the map obtained by inputting $u_{\mathcal{C}'}$ in $$H(\widehat{tel}(\mathcal{C}'))\otimes H(\widehat{tel}(\mathcal{C}))\to H(\widehat{tel}(\mathcal{D}))$$ is equal to the map induced by $f$.
\end{lem}

\begin{proof}It is easy to see that $U_{\mathcal{C}'}$ is equivalent to a strict morphism. Therefore, using Lemma \ref{lemstrictunit}, we can assume that $U_{\mathcal{C}'}$ is strict.

Then using functoriality as in Lemma \ref{lemmafunctorial}, we can reduce to the case $C'=U$. We finish the proof if we can prove that the composition of all the maps below

\begin{itemize}
\item $H^*(\widehat{tel}(\mathcal{C}))\to H^*(\widehat{tel}(\mathcal{U}))\otimes H^*(\widehat{tel}(\mathcal{C}))$
\item $H^*(\widehat{tel}(\mathcal{U}))\otimes H^*(\widehat{tel}(\mathcal{C}))\to H^*(\widehat{tel}(\mathcal{U})\otimes \widehat{tel}(\mathcal{C}))$;
\item $H^*(\widehat{tel}(\mathcal{U})\otimes \widehat{tel}(\mathcal{C}))\to H^*(\widehat{tel(\mathcal{U})\otimes tel(\mathcal{C})})$;
\item $H^*(\widehat{tel(\mathcal{U})\otimes tel(\mathcal{C})})\to H^*(\widehat{tel}(\mathcal{U}\otimes \mathcal{C}))$;
\end{itemize}
is the same as the map $H^*(\widehat{tel}(\mathcal{C}))\to H^*(\widehat{tel(\mathcal{U}\otimes \mathcal{C})})$ induced by $\mathcal{C}= \mathcal{U}\otimes \mathcal{C}$. This is straightforward.
\end{proof}

\begin{cor} 
	\label{cor:real_iso}
	In the seting of Lemma \ref{lem:realis_h}, let $\cC=\cC'$. Assume that the map  $$H(\widehat{tel}(\mathcal{C}))\otimes\Lambda\to H(\widehat{tel}({\mathcal{D})})\otimes \Lambda$$ induced by $f$ is an isomorphism.  Then, $H(\widehat{tel}(\cC))\otimes\Lambda=0$ if and only if $u_{\mathcal{C}}$ is torsion, i.e.~$u_{\cC}=0$ in $H(\widehat{tel}(\mathcal{C}))\otimes\Lambda$.\end{cor}

\section{Contact Fukaya trick}\label{secfuk}

\subsection{Necks and admissible functions}

Let $(M,\omega)$ be a closed symplectic manifold and  $\Sigma\times [1  -\alpha,s+\alpha]\subset M$ be a neck as in Definition \ref{defneck}.

Let $U_1$ be the convex filling of $\Sigma=\Sigma\times \{1\}$ and define
$$U_s=U_1\cup (\Sigma\times [1,s]).
$$

Here is the plan of what follows.
The idea is to fix a specially chosen diffeomorphism $g\co M\to M$ which is identity outside of the neck, and acts only on the coordinate $r$ strictly monotonically inside the neck, taking $U_1$ to $U_s$. The next step is to construct cofinal families of Hamiltonians $f_i$ for $U_1$ and $F_i$ for $U_s$, whose Hamiltonian vector fields are related by $g$. Also note that under $g$ the push forward of a compatible almost complex structure, which is cylindrical in the neck, is still tame for the original symplectic structure. We will also make sure that it is compatible near the $1$-periodic orbits. This construction requires  care because $g$ is not a symplectomorphism. 

If we use the above cofinal families and almost complex structures, then $g$ gives a bijection between the Floer solutions contributing to the Floer 1-rays for $U_1$  and $U_s$. However, the topological energies of a Floer cylinder and its $g$-image are different. This is the content of what we call the {\it contact Fukaya trick}\footnote{The name comes from Fukaya's trick that is commonly used in Family Floer theory, see for example Section 8.1 of \cite{Yuan}. We hope that the similarities will be clear to the reader.}, see Proposition~\ref{prop:fukaya_trick}.
In particular, $g$ does not necessarily give rise to a quasi-isomorphism of the corresponding completed telescopes. However, we show that it does give  a quasi-isomorphism in the $c_1(M)=0$, index bounded case as claimed in Proposition \ref{prop:index_bd_collar_invt}. The ``concave filling" case is analogous, so we focus on the convex filling case we just explained.

We begin with the diffeomorphism $g$.
Consider a compactly supported diffeomorphism 
$$
g(r)\co [1-\alpha, s+\alpha]\to [1-\alpha, s+\alpha],\quad g(1)=s,
$$
with the following properties:
$$
\begin{cases}
g(r)\text{ is monotone onto } [1-\alpha,s-\epsilon],& r\in[1-\alpha,1-\epsilon]\\
g(r)=sr,& r\in[1-\epsilon,1]\\
g(r)\text{ is monotone onto }[s,
\tilde s(1+\epsilon)],&r\in [1,1+\epsilon]\\
g(r)=\tilde s r,& r\in[1+\epsilon,1+2\epsilon]\\
g(r)\text{ is monotone onto }[\tilde s(1+2\epsilon),s+\alpha],&r\in [1+2\epsilon,  s+\alpha].
\end{cases}
$$
Recall that $s$ was fixed above; $\tilde s$ and $\epsilon$ are free parameters, $\tilde s>s$.
We fix $\epsilon<\min(\alpha,s)$ to be sufficiently small, and define $\tilde s$
by the equation
\begin{equation}
\label{eq:tilde_s}
\epsilon=\tilde s (1+\epsilon)-s.
\end{equation}
In other words, $\tilde s=(s+\epsilon)/(1+\epsilon)$.
The reason for this choice will be seen later.

Let $\phi$ be the diffeomorphism of $M$ induced by  $g$ on the neck, extended by the identity outside of the neck.

Let $\cR$ be pullback of the Reeb vector field on $\Sigma$ under the projection $[1,s+\alpha]\times \Sigma\to \{1\}\times\Sigma$.
Then for an arbitrary function $f(r)$ of the collar coordinate, its Hamiltonian vector field is
\begin{equation}
\label{eq:Ham_f_neck}
X_{f(r)}=f'(r)\cdot \cR.
\end{equation}

Consider a $C^\infty$ function $f(r)$, $r\in[1-\alpha,s+\alpha]$, with the following properties:
$$
\begin{cases}
f(r)=cr+b,& r\in[1-\alpha,1-\epsilon]\\
f(r)\text { is monotone and has total increase $\delta$,}& r\in[1-\epsilon,1],\\
f(1)=0,\\
f(r)=Kr  +c_K,&r\in [1,1+\epsilon]\\
f(r)\text { is monotone and has total increase $\delta$,}& r\in[1+\epsilon,1+2\epsilon]\\
f(r)=cr+B,&r\in [1+2\epsilon, s+\alpha].
\end{cases}
$$
We  think of $\delta$ and the slopes $c,K$ as free parameters, whereas the constants $b,c_K,B$ are determined by continuity. Total increase $\delta$ means, for instance in the first case, that $f(1)-f(1-\epsilon)=\delta$. We call such functions of the coordinate $r$ {\it profile functions}.

Note that for any given $\delta>0$, for sufficiently small $c$ one has:
\begin{equation}
\label{eq:osc_collar}
f(1)-f(1-\alpha)<2\delta\quad\text{and}\quad
f(s+\alpha)-f(1+\epsilon)<2\delta.
\end{equation}

Now fix some $\delta>0$.
We call a time-independent Hamiltonian $H\co M\to\bR$ {\it admissible with profile $f$} if it: 
\begin{itemize}
	\item 
	restricts to $f(r)$ on the neck $\Sigma\times [1  -\alpha,s+\alpha]$, 
	
	\item
has non-degenerate periodic orbits outside of the neck, for convenience all assumed to be constant. 

\end{itemize}

Moreover, we demand that the profile $f(r)$ of an admissible Hamiltonian satisfies:
\begin{itemize}
	\item  $K$ is not a period of $\cR$;
	\item $c$ is smaller than the periods of the periodic orbits of $\cR$;
	\item if $f'(r)$ is equal to a period of $\cR$, then $f''(r)\neq 0$;
	\item $(s-1+\alpha)c<c(1-\alpha)+b$ and $c(1-\alpha)+b+\epsilon>0$.
\end{itemize}

Note that as a result of our assumptions, the non-constant periodic orbits of $H$ are transversely non-degenerate. We can make time-dependent perturbations to $H$ supported in arbitrarily small neighborhoods of its non-constant periodic orbits, which are arbitrarily $C^{\infty}$-small, so that the perturbed Hamiltonian is non-degenerate. Moreover, for each non-constant periodic orbit, the perturbed Hamiltonian has two periodic orbits with periods (the integral of the Liouville form) arbitrarily close to the period of the original orbit and indices that differ from the Conley-Zehnder index of the corresponding Reeb orbits by at most $1$. We will ignore these perturbations in the narrative and only talk about them when there is something necessary to point out. We will refer to them as harmless perturbations.

\subsection{Matching profiles }
Recall the diffeomorphism $\phi\co M\to M$ defined above, determined by the function $g(r)$. For a profile function $f(r)$, consider the pushforward of its Hamiltonian vector field under $\phi$:
\begin{equation}
\label{eq:phi_X}
\phi_*X_{f(r)}=f'(g^{-1}(r))\cdot \cR.
\end{equation}

%


We claim  that $\phi_*X_{f(r)}$ is a Hamiltonian vector field of the smooth function $F(r)$:
$$
\begin{cases}
F(r)=cr+\tilde b,& r\in[1-\alpha,s(1-\epsilon)]\\
F(r)=sf(r/s),& r\in[s(1-\epsilon),s]\\
F(r)=Kr+\tilde{c}_K,&r\in [s,\tilde s(1+\epsilon)]\\
F(r)=\tilde s f(r/\tilde s)+\tilde d,& r\in[\tilde s(1+\epsilon),\tilde s(1+2\epsilon)]\\
F(r)=cr+\tilde{B},&r\in [\tilde s(1+2\epsilon),s+\alpha].
\end{cases}
$$ 
The constants $\tilde b,\tilde{c}_K, \tilde d,\tilde B$ are determined by continuity. See Figure~\ref{fig:neck_push}.

\begin{rem}
Clearly, $F(r)$ is determined up to an additive constant, and we have chosen the normalization $F(s)=0$. This is seen from the fact that $f(1)=0$.
\end{rem}

\begin{figure}[h]
	\includegraphics[]{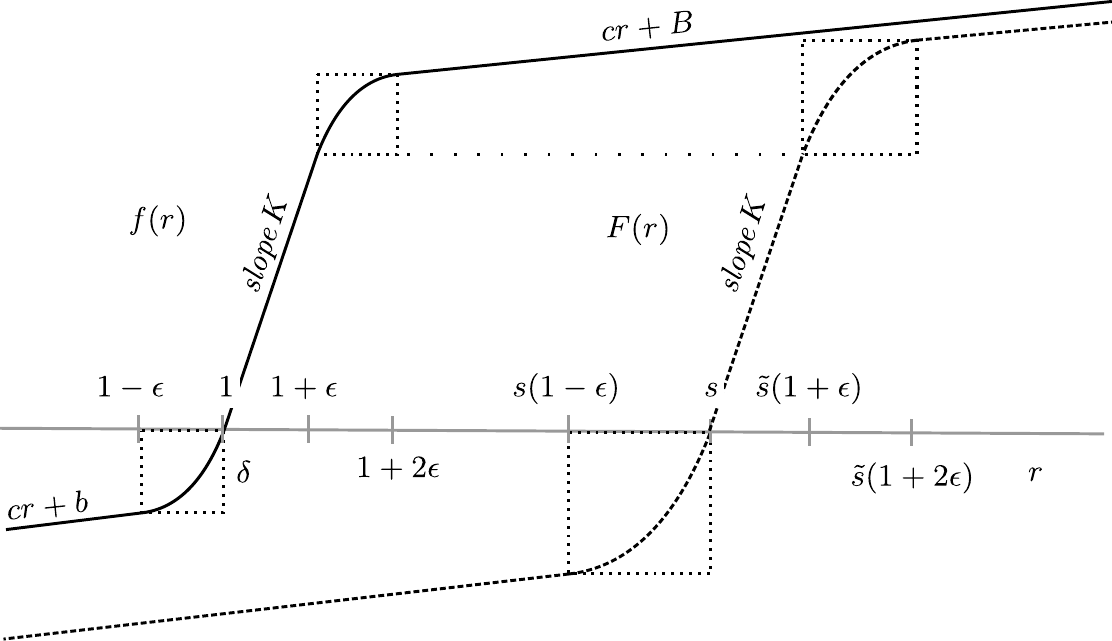}
	\caption{A profile function $f$ (solid plot), and a matching function $F$ (dotted plot).}
	\label{fig:neck_push}
\end{figure}

We claim that 
\begin{equation}
\label{eq:X_F}
\phi_*X_{f(r)}=X_{F(r)}.
\end{equation}
Indeed, in view of (\ref{eq:phi_X}) and (\ref{eq:Ham_f_neck}) this is equivalent to:
$$
F'(r)=f'(g^{-1}(r)),
$$
which is easy to check.
We say that this function $F(r)$ {\it matches} $f(r)$.

\begin{lemma}
	\label{lem:match_dif_1}
	Fix a constant $C>0$. Then for all $K>0$, and all sufficiently small $c,\delta>0$, the following holds.
	Suppose $f(r)$ is a profile with sufficiently small parameters $c,\delta$ and arbitrary $K$. Let $F(r)$ match $f$. It holds that:
	\begin{equation}
	\label{eq:F_f_dif}
	|F(1-\alpha)-f(1-\alpha)|<C,\quad |F(s+\alpha)-f(s+\alpha)|<C,
	\end{equation}
	and
	\begin{equation}
	\label{eq:F_f_curvy_region}
	|F(g(r))-f(r)|<C\quad\text{for all}\quad\text{for}\quad r\in [1-\epsilon,1]\cup[1+\epsilon,1+2\epsilon].
	\end{equation}
\end{lemma}

\begin{rem}
	The important point is that the above bounds do not depend on $K$ which can be arbitrarily large.
Note that as $K\to+\infty$, $f(r)\to+\infty$ on $(1,s+\alpha]$ and $F(r)\to+\infty$ on $(s,s+\alpha]$.
\end{rem}

\begin{proof}
	Recall that $f(1)=F(s)=0$. By construction,
	$$
	\begin{array}{l}
	f(1-\alpha)=-\delta-c(\alpha-\epsilon),\\ F(1-\alpha)=-\delta s-c(s(1-\epsilon)-(1-\alpha)).
	\end{array}
	$$
	Both values do not depend on $K$, and are small when $c,\delta$ are small.

	Recall that according to (\ref{eq:tilde_s}), $\epsilon=\tilde s (1+\epsilon)-s$. This choice has the following crucial significance. By construction, $f$ has slope $K$ on the segment $[1,1+\epsilon]$ of length $\epsilon$, so it has total increase $K\epsilon$ on this segment. Again by construction, $F$ has slope $K$ on the segment $[s,\tilde s(1+\epsilon)]$, which is of length $\epsilon$  by (\ref{eq:tilde_s}). So
	
	$$
	\begin{array}{l}
	f(1+\epsilon)=K\epsilon,\\
	F(\tilde s(1+\epsilon))=K\epsilon.
	\end{array}
	$$
	Continuing the computation,
$$
	\begin{array}{l}
	f(s+\alpha)=K\epsilon+\delta+c(s+\alpha-(1+\epsilon)),\\
	F(s+\alpha)=K\epsilon +\delta\tilde s+c(s+\alpha-\tilde s(1+2\epsilon)).
	\end{array}
$$
We see that the difference is small, if $\delta,c$ are small. We have proved (\ref{eq:F_f_dif}).

To prove (\ref{eq:F_f_curvy_region}),  one easily checks:
$$
\begin{array}{l}
f([1-\epsilon,1])=[-\delta,0],\\
F(g([1-\epsilon,1]))= [-\delta s,0],
\end{array}
$$
and
$$
\begin{array}{l}
f([1+\epsilon,1+2\epsilon])=[K\epsilon,K\epsilon+\delta],\\
F(g([1+\epsilon,1+2\epsilon]))\subset [K\epsilon, K\epsilon+\delta\tilde s].
\end{array}
$$
Now (\ref{eq:F_f_curvy_region}) easily follows.
\end{proof}	

Let $H$ be an admissible Hamiltonian with profile $f$, and let $F$ match $f$. We say that a Hamiltonian $\tilde H$ on $M$ {\it matches $H$} if:
\begin{itemize}
	\item $\tilde H$ restricts to $F(r)$ on the neck $\Sigma\times [1  -\alpha,s+\alpha]$, and 
	
	\item on each of the two connected components of the complement of the neck, $H-\tilde H$ is constant. Explicitly,
	\begin{equation}
	\label{eq:H_h_dif_col_1}
	\tilde H=H+F(1-\alpha)-f(1-\alpha) \quad \text{on}\quad U_{1-\alpha},
	\end{equation}
	and
	\begin{equation}
	\label{eq:H_h_dif_col_2}
	\tilde H=H+F(s+\alpha)-f(s+\alpha) \quad \text{on}\quad M\setminus U_{s+\alpha}.
	\end{equation}
	
\end{itemize}

\begin{lemma}
	\label{lem:phi_H}
If $\tilde H$ matches $H$, then $	\phi_*X_{H}=X_{\tilde H}$.
\end{lemma}

\begin{proof}
	On the neck, the claim follows from the analogous fact about $F(r)$ and $f(r)$,  see Equation~(\ref{eq:phi_X}). 
Away from the neck, $\phi$ is the identity and $\tilde H-H$ is locally constant, so the claim holds there too. 
\end{proof}	

This Lemma can be easily modified to include harmless perturbations of $\tilde H$ and $H$. The main point is that in the regions where $H$ has its non-constant periodic orbits, the diffeomorphism $\phi$ is a symplectomorphism up to multiplication by a constant. Therefore, if the vector fields of the Hamiltonian perturbations are moved with $\phi$, they will stay Hamiltonian perturbations and act as harmless perturbations to $\tilde H$. Of course, there are various smallness assumptions that we are not explicitly writing down.

\subsection{Cofinal families}
Consider a neck $\Sigma\times [1  -\alpha,s+\alpha]\subset M$. As above, fix
$\epsilon>0$, a diffeomorphism $g$ of $[1-\alpha,s+\alpha]$ and the induced diffeomorphism $\phi$ of $M$. 
Fix monotonically decreasing sequences of positive numbers $\delta_i\to 0$, $c_i\to 0$, $\beta_i\to 0$, and a monotonically increasing sequence $K_i\to +\infty$. For each $i$, choose a profile  function $f_i(r)$ with the given parameters $\epsilon ,\delta_i,c_i,K_i$. We require that $f_i(r)\le f_{i+1}(r)$ for all $r\in[1-\alpha,s+\alpha]$.
Note that
$$
f_i(r)-\beta_i \to
\begin{cases}
0,& r\le 1\\
+\infty,&r>1
\end{cases} 
$$
The same is true about $f_i(r)$ without the summand $
\beta_i$. We have subtracted $\beta_i$ to make sure that $f_i(r)-\beta_i<0$ for $r\le 1$ as required in the definition of a cofinal family (recall that $f_i(1)=0$).
Let $F_i(r)$ match $f_i(r)$, then
$$
F_i(r)-\beta_i \to
\begin{cases}
0,& r\le s\\
+\infty,&r>s
\end{cases}.
$$
Let $H_i$ be an admissible Hamiltonian with profile $f_i$, and $\tilde H_i$ be an admissible Hamiltonian with profile $F_i$. Then $\{H_i\}$ are a cofinal family for $U_1\subset M$, and $\{\tilde H_i\}$ are a cofinal family for $U_s\subset M$.

\begin{lemma}
	Let $J$ be a tame almost complex structures  which is cylindrical on
	\begin{equation}
	\label{eq:sigma_times_int_1}
	\Sigma\times \left([1-\alpha,s+\alpha]\right),
	\end{equation}
	and compatible with $\omega$ on
	
Then $\phi_*J$ is tame, and moreover, it is compatible on \begin{equation}
	\label{eq:sigma_times_int_2}
	\Sigma\times \left([s(1-\epsilon)s,s]\cup[\tilde{s}(1+\epsilon),\tilde{s}(1+2\epsilon)]\right).
	\end{equation}
\end{lemma} 

\begin{proof}
	Since $\phi$ is identity outside of the neck, it suffices to check this on the neck. First, recall that $J$ is cylindrical and $\phi$ only depends on the radial coordinate on the region (\ref{eq:sigma_times_int_1}), so $\phi_*J$ is tame on the image of that region. 
	
	Second, recall that $\phi$ scales the neck coordinate $r$ by a fixed factor  on each of the two components of the region (\ref{eq:sigma_times_int_2}). Because the neck coordinate $r$ is the exponential of the Liouville coordinate, $\phi$ acts by a shift of the Liouville coordinate on (\ref{eq:sigma_times_int_2}). Hence, $\phi$ is conformally symplectic onto the image  of (\ref{eq:sigma_times_int_2}). So  $\phi_*J$ is still cylindrical on that region, in particular it is compatible.
\end{proof}

Let $f_i$ and $F_i$ be as above, choose Hamiltonians $\{H_i\}$ with profile $f_i$, and let $\tilde H_i$ match $H_i$ (so that $\tilde H_i$ has profile $F_i$).
We can arrange that $H_i$ is a cofinal family for $U_1$, and $\tilde H_i$ is a cofinal family for $U_s$.

Next, choose monotone homotopies between the $f_i$ within the space of profile functions. Extend them to monotone homotopies between the $H_i$.
Consider the matching monotone homotopies between the $F_i$, and extend them to monotone homotopies between the $\tilde H_i$. This produces Floer 1-rays
$\cC$ (for $U_1$ using $\{H_i\}$) and $\cC'$ (for $U_s$ using $\{\tilde H_i\}$).

\begin{lemma}
	\label{lem:phi_bij}
The diffeomorphism $\phi$ takes solutions of the Floer equation contributing to $CF^*(H_i,J_i)$ bijectively to solutions for $CF^*(\tilde H_i,\phi_*J_i)$. The analogous holds for continuation maps.

In other words, $\phi$ takes all solutions contributing to the structure maps of  $\cC$ bijectively onto Floer solutions contributing to the structure maps of $\cC'$.
\end{lemma}

\begin{proof}
	This is a reformulation of Lemma~\ref{lem:phi_H}. The analogous property is true for matching monotone homotopies  constructed using matching profile functions.
\end{proof}

\subsection{Contact Fukaya trick}
Consider cofinal families $H_i$ for $U_1$ and $\tilde H_i$ for $U_s$, with Floer 1-rays $\cC$ and $\cC'$ as above.
We will translate Lemma~\ref{lem:phi_bij} to a comparison between $\cC$ and $\cC'$.

A generator of $\cC$ is a 1-periodic orbit of $H_i$ for some $i$. These generators can be of two types:

\begin{itemize}
	\item constant orbits in $U_{1-\alpha}$ and $M\setminus U_{s+\alpha}$,
	\item non-constant orbits in $\Sigma\times ([1-\epsilon,1]\cup [1+\epsilon,1+2\epsilon])$.
\end{itemize}
Generators of $\cC'$ are 1-periodic orbits of $\tilde H_i$ for some $i$; they are of two types as above, with the difference that non-constant orbits belong to $\Sigma\times ([s(1-\epsilon),s]\cup [\tilde s(1+\epsilon),\tilde s(1+2\epsilon)])$. 

By Lemma~\ref{lem:phi_H}, $\phi$ induces a bijection between the generators of $\cC$ and $\cC'$, $\gamma\mapsto \phi(\gamma)$.
Lemma~\ref{lem:phi_bij} states that there is a bijection at the level of structure maps; however, note that the topological energies of $u$ and $\phi(u)$ are different. The proposition below expresses this difference.

\begin{prop}[Contact Fukaya trick]
	\label{prop:fukaya_trick}
Consider the natural bijection between the generators of $\cC$ and $\cC'$ discussed above: $\gamma\mapsto \tilde\gamma\coloneqq \phi(\gamma)$. It upgrades to a  strict isomorphism between the chain complexes
$$
\Phi\co tel(\cC)\otimes \Lambda \to tel(\cC')\otimes \Lambda,$$ 
which has the following form:
$$
\Phi\co \gamma\mapsto  T^{\Delta(\gamma)}\cdot \tilde\gamma,
$$
for both generators of $ tel(\cC)$ corresponding to $\gamma$.
The numbers $\Delta(\gamma)\in \bR$ are defined in (\ref{eq:delta_gamma}) below, and satisfy:
\begin{equation}
\label{eq:delta_O1}
\Delta(\gamma)=\omega(C(\gamma,\tilde \gamma))+O(1).
\end{equation}
Here the first summand for $\gamma$ a non-constant orbit is the area of the trivial cylinder $C(\gamma,\tilde \gamma)\subset X$ spanning $\gamma$ and $\tilde \gamma$, and contained in the neck; for $\gamma$ a constant orbit it is zero. The second summand $O(1)$ is a quantity depending on $\gamma$ which is uniformly bounded in  absolute value across all generators.
\end{prop}

\begin{proof}
Consider a Floer solution $u$ contributing to $\cC$ with asymptotic orbits $\gamma_{in},\gamma_{out}$.
By Lemma~\ref{lem:phi_bij}, $\phi(u)$ is a Floer solution contributing  to $\cC'$.
Recall the definition of topological energy  (\ref{eq:E_top}):
$$
E_{top}(u)=\int_{S^1}\gamma_{out}^* H_{out}\,dt-\int_{S^1}\gamma_{in}^* H_{in}\,dt+\omega(u).
$$
Define, for a $1$-periodic orbit $\gamma$ of $H_i$:
\begin{equation}
\label{eq:delta_gamma}
\Delta(\gamma)=\left(\int_{S^1} \tilde \gamma^* \tilde H_i-\gamma^*H_i\right)+\omega(C(\gamma,\tilde\gamma)).
\end{equation}
Then for a Floer solution $u$ as above one has:
$$E_{top}(\phi(u))-E_{top}(u)=\Delta(\gamma_{out})-\Delta(\gamma_{in}),$$ because of the obvious diffeotopy from $\phi$ to identity and Stokes theorem.
Recall that a Floer solution  $u$ contributes to $\cC$ as follows:
$$
\gamma_{in}\mapsto T^{E_{top}(u)}\gamma_{out},
$$
while $\phi(u)$ contributes to $\cC'$ as follows:
$$
\tilde \gamma_{in}\mapsto T^{E_{top}(u)+\Delta(\gamma_{out})-\Delta(\gamma_{in})}\tilde \gamma_{out},
$$
or equivalently
$$
T^{\Delta(\gamma_{in})}\tilde \gamma_{in}\mapsto T^{\Delta(\gamma_{out})}T^{E_{top}(u)}\tilde \gamma_{out}.
$$
This can be rewritten as
$$
\Phi(\tilde \gamma_{in})\mapsto T^{E_{top(u)}}\Phi(\tilde \gamma_{out}).
$$
So $\Phi$ is a chain isomorphism.
To show (\ref{eq:delta_O1}), in view of (\ref{eq:delta_gamma}) it remains to show that there exists a constant $C$ such that:
\begin{equation}
\label{eq:diff_H}
\left| \int_{S^1} \tilde \gamma^* \tilde H_i-\gamma^*H_i\right|<C
\end{equation}
for all $1$-periodic orbits $\gamma$ of $H_i$, for all $i$.
We claim that this follows from Lemma~\ref{lem:match_dif_1}. 

To see why, assume first that $\gamma$ is a constant orbit, then it belongs to  $U_\alpha$ or $E_{s+\alpha}$, and $\tilde \gamma=\gamma$. The above difference is simply $\tilde H(\gamma)-H(\gamma)$, where $\gamma$ is seen as a point in $X$. Then (\ref{eq:diff_H}) follows from (\ref{eq:F_f_dif}), (\ref{eq:H_h_dif_col_1}) and (\ref{eq:H_h_dif_col_2}).

Now assume $\gamma$ is a non-constant orbit, then $\gamma$ belongs to the neck where $H_i=f_i$, $\tilde H_i=F_i$. By construction of the $f_i$, $\gamma$ belongs to the region of the neck where $r\in [1-\epsilon,1]\cup[1+\epsilon,1+2\epsilon]$. Here (\ref{eq:F_f_curvy_region}) applies to guarantee (\ref{eq:diff_H}).
\end{proof}

It is clear that the harmless perturbations are harmless from the viewpoint of this Lemma as well.

\subsection{Index boundedness}
Let us recall the setup of this section. We started with a neck $\Sigma\times [1-\alpha,s+\alpha]\subset M$ in a symplectic manifold $M$. We denoted by $U_1$ the domain bounded by $\Sigma\times\{1\}$, and took $U_s=U_1\cup (\Sigma\times [1,s])$.
 We constructed Floer 1-rays $\cC,\cC'$ which compute $SH^*_M(U_1)$ and $SH^*_M(U_s)$, in particular:
 $$
 H^*(\widehat{tel}(\cC))\otimes\Lambda=SH^*_M(U_1, \Lambda),\quad H^*(\widehat{tel}(\cC'))\otimes\Lambda=SH^*_M(U_s,\Lambda). 
 $$
 We explained how to carefully choose these Floer 1-rays so that they satisfy Proposition~\ref{prop:fukaya_trick}. 
 
At this point we would like to remind the reader Remark \ref{remanalytic}. With that in mind, in general, Proposition~\ref{prop:fukaya_trick} does not give any straightforward relationship between $SH^*_M(U_1,\Lambda)$ and $SH^*_M(U_s,\Lambda)$. This is because the map $\gamma\mapsto  T^{\Delta(\gamma)}\cdot \tilde\gamma$ and/or its inverse may not extend to the completions, as they might not be continuous, or equivalently bounded.

 We now assume that the neck is index bounded, i.e. for every integer $i$, the periods of the periodic orbits of $\Sigma$ of Conley-Zehnder index $i$ have a uniform upper and lower bound.

\begin{proof}[Proof of Proposition \ref{prop:index_bd_collar_invt}]
In our current notation, the statement is equivalent to
	$$
	SH^*_M(U_1,\Lambda)\cong SH^*_M(U_s,\Lambda). 
	$$

Note that there exists a constant $a>0$ such that for any non-constant $1$-periodic orbit $\gamma$, the quantity $\omega(C(\gamma,\tilde\gamma))$ is bounded by $a$ times the period of the Reeb orbit corresponding to $\gamma$ in $\Sigma$. 

Therefore, the index bounded condition implies that the map from Proposition~\ref{prop:fukaya_trick} and its inverse are in fact both continuous, and therefore they can be extended to the completions. Since this extension is a functorial operation, they are still chain maps and strict inverses of each other, which finishes the proof.
\end{proof}

\section{Product and unit via raised cohomology}
\label{sec:prod_unit}
\subsection{Raised cohomology and perturbation spaces}
\begin{defin}
	\label{dfn:raised_sh}
Let  $K\subset M$ be a compact subset of a symplectic manifold and $\epsilon>0$. Define the raised symplectic cohomology
$$SH^*_M(K,\epsilon)$$ 
analogously to $SH^*_M(K)$ using increasing families $H_i$ which are cofinal for $\{H:H|_{S^1\times K}<\epsilon\}$, equivalently,
$$
H_i(x)\xrightarrow[i\to+\infty]{}\begin{cases}
\epsilon,& x\in K,\\
+\infty,& x\notin K.
\end{cases}
$$
Moreover, we will always assume $H_i>0$ everywhere on $M$ for such cofinal families; note that this would not be possible with the usual definition of $SH^*_M(K)$.
\end{defin}

Let $\cM_{k;1}$ be the moduli space of genus~zero Riemann surfaces with $k$ input punctures and one output puncture, modulo automorphisms. We assume that they are equipped with compatible choices of cylindrical ends. We will only be interested in the cases $k=0,1,2$, when $\cM_{k;1}$ consists of a unique element, a curve $C\in \cM_{k;1}$.

Fix weight parameters 
$$\bw=(w^i>0)_{i=1,\ldots, k}\cup (w^0>0),$$ one for each input and output (so that in the case $k=0$ we only have $w^0$). The superscripts are used as indices here. Also fix
$$
\bH=(H^i)_{i=1,\ldots, k}\cup (H),
$$
a collection of  time-dependent everywhere positive Hamiltonians on $M$, one for each input and output.

Define 
\begin{equation}
\label{eq:S}
\cS_{k;1}(\bw,\bH)=\{
(\alpha, H)\}
\subset \Omega^1(C)\times  C^\infty(C\times M)\end{equation}
 to be the subset consisting of all $\alpha,H$ with the following properties:
 $$
\left\{
\begin{array}{l}
d\alpha\ge 0,\ H\ge 0,\\
d_{C\times M}(H\alpha)\mid_{C\times\{x\}}\ge 0, \text{ for every }x\in M\\
\alpha\equiv w^idt\text { on cylindrical ends,}\\
H\equiv H^i(t)\text { on cylindrical ends.}\\
\end{array}
\right.
$$

Given such a pair $(H,\alpha)$, we obtain a Hamiltonian vector field valued one form $X_H\otimes \alpha$ on $C$. This association factors through the canonical map $$\Omega^1(C)\times  C^\infty(C\times M)\to \Omega^1(C)\otimes_{C^\infty(C)}  C^\infty(C\times M).$$ The image of this map captures the data that is actually used in the Floer theoretic constructions. The first two requirements ensure that the topological energies of the solutions of the Floer equation $$(du-X_H\otimes \alpha)^{0,1}=0,$$ with respect to any domain dependent compatible almost complex structure, are non-negative. Let us demonstrate this.

Let $(C, j) \in \cM_{k;1}$, and $(\alpha, H)\in \cS_{k;1}(\bw,\bH)$ for some weights $\bw$ and time-dependent Hamiltonians $\bH$. Consider the trivial bundle $\pi:C \times M \to C$ with the connection $2$-form $$\Omega:=pr_M^*\omega+d_{C\times M}(H\alpha).$$ Recall the one to one correspondence between $C$-dependent almost complex structures $J_M$ on $M$ and almost complex structures $J^{tot}$ on $C \times M$ which are split for the horizontal subbundle defined by $\Omega$, and which make $\pi$ a $(J^{tot},j)$-holomorphic map. Under this correspondence, solutions $u$ of the Floer equation $(du-X_H\otimes \alpha)^{0,1}=0$ for $J_M$ are in turn in one to one correspondance with $(j,J^{tot})$-holomorphic sections $\tilde{u}=(u,id)$ of $\pi$. Moreover, we have the equality: $$E_{\text{top}}(u)=E_{\text{top}}(\tilde{u}),$$ by definition.

We now claim that our assumption $d_{C\times M}(H\alpha)\mid_{C\times\{x\}}\geq 0$ implies that $\Omega$ and $J^{tot}$ form a semi-tame pair for any semi-tame $J_M$, namely that $\Omega(w, J^{tot}w)\geq 0$, for any tangent vector $w$. If $w$ is a vertical vector than this is satisfied by the semi-tameness of $J_M$. If we decompose $w$ into its vertical and horizontal components $w=w_v+w_h$, by linearity, we obtain $$\Omega(w, J^{tot}w)=\Omega(w_v, J^{tot}w_v)+\Omega(w_v, J^{tot}w_h)+\Omega(w_h, J^{tot}w_v)+\Omega(w_h, J^{tot}w_h).$$Noting that the second and third terms are zero by the definition of the horizontal bundle and the splitness of $J^{tot}$, we reduce to showing the claim for when $w$ is a horizontal vector. Note that all horizontal vectors are lifts of tangent vectors from the base $C$.

A simple computation shows that for any $v\in T_pC$, its unique horizontal lift to $(p,x)$ is given by $$v+\theta(v)X_H(p,x)$$ via the splitting $T_{(p,x)}(C \times M) =T_pC \oplus T_xM .$ The desired result follows by plugging in all of these in the semi-tameness equation and using $d_{C\times X}(H\alpha)\mid_{C\times\{x\}}\geq 0$.

This implies that any $J^{tot}$-holomorphic curve inside $C\times M$ has non-negative geometric energy, which in turn can be shown to equal topological energy by the standard argument. The upshot is that if $u$ is a solution of the Floer equation, then $E_{\text{top}}(u)\geq 0$, as desired.
\\


It will be convenient to also refer to the underlying space of 1-forms:
\begin{equation}
\label{eq:Omega_space}
\Omega_{k;1}(\bw)=\{\alpha\in \Omega^1(C):\ d\alpha\ge 0 \text{ on }C,\ \alpha\equiv w^idt\text{ on cylindrical ends}\}.
\end{equation}
Clearly, there is a projection 
$$
\pi\co \cS_{k;1}(\bw,\bH)\to \Omega_{k;1}(\bw).
$$

We shall use the spaces $\cS_{k;1}(\bw,\bH)$ to define (raised) symplectic cohomology, products, and units. Table~\ref{tab:weights} may be helpful to keep track of the definitions.

\begin{table}[h]
	
\begin{tabular}{|c|c|c|}

	\multicolumn{3}{c}{Floer differential and continuation maps for Floer 1-rays,}
	\\
	\multicolumn{3}{c}{Extrinsic continuation maps $c_{\epsilon_1,\epsilon_2}$}
	\\	
	\hline
	$k=1$
	&
	$w^1=w^0=1$
	& $\cS_c(\bH)$
	\\
	\hline
	\end{tabular}
	\medskip

	\begin{tabular}{|c|c|c|}	
	\multicolumn{3}{c}{Another version of extrinsic continuation, $\tilde c_{\epsilon,2\epsilon}$}
	\\
	\hline
	$k=1$
	&
	$w^1=1$, $w^0=2$
	&
	$\cS_{\tilde c}(\bH)$
	\\
	\hline
		\end{tabular}
\medskip

	\begin{tabular}{|c|c|c|}	

	\multicolumn{3}{c}{Unit $1_\epsilon\in SH^*_M(K,\epsilon)$}
	\\
	\hline
	$k=0$
	&
	$w^0=1$
	&
	$\cS_u(\bH)$
	\\
	\hline
	\end{tabular}
\medskip

	\begin{tabular}{|c|c|c|}	
	
	\multicolumn{3}{c}{Product $*_\epsilon\co SH^*_M(K,\epsilon)^{\otimes 2}\to SH^*_M(K,2\epsilon)$}
	\\
	\hline
	$k=2$
	&
	$w^1=w^2=1$, $w^0=2$
	&
	$\cS_p(\bH)$
	\\
	\hline
\end{tabular}
	\medskip
	
\caption{Weights $\bw$ used in the definition of different operations on raised relative symplectic cohomology. The last column gives short names to the spaces $\cS_{k;1}(\bw,\bH)$ with the specified $\bw$.}
\label{tab:weights}
\end{table}

\begin{rem}\label{remgluing}
There are natural gluing operations on these perturbations spaces. Their existence is important to ensure the properties of the operations we are going to define using these perturbations.
For example, there are maps
$$
\#_\rho\co \cS_c\sqcup \cS_c\to \cS_c
$$
which glue perturbation data on the strips using gluing parameter $\rho\gg 0$ in two possible orders, whenever the Hamiltonians at the glued cylindrical ends match. 
As another example, there is a map
$$
\#_\rho\co \cS_u\sqcup \cS_p\to \cS_c,
$$
gluing the  unit datum to the product datum at, say, the second input, which will be used to show that  $-*_\epsilon 1_\epsilon$ is chain homotopic to $c_{\epsilon,2\epsilon}$. There are other natural gluing maps like this. Rather than writing them all down, we summarize the main point and do not mention it further: all gluing operations needed below can be performed staying within the spaces (\ref{eq:S}) of perturbation data, unless explicitly mentioned otherwise. A full treatment would have introduced families of choices parametrized by certain manifolds with corners where near the boundaries existence of ``gluing coordinates" is assumed. We would then frame contractibility as families defined on boundaries of these manifolds with corners admitting extensions to the interior. This is standard especially for our rudimentary needs with very explicit diagrams containing only homotopies of homotopies. We will omit further mention of this and only prove the main point, which is that our space of unbroken choices is contractible.
\end{rem}

The first  thing we need to understand is the conditions under which the spaces~(\ref{eq:S}) are non-empty and contractible. For two time-dependent Hamiltonians $H_1,H_2\co S^1\times X\to \bR$, let us write
$$
\begin{array}{c}
H_1\succeq H_2 \textit{ if }
\exists\textit{ a time-independent } H^\textit{ref}\co X\to\bR\textit{ s.~t.~}
\\
H_1(t,x)\ge H^\textit{ref}(x)\ge H_2(t,x) \textit{ for every }(t,x)\in S^1\times X.
\end{array}
$$

\begin{lemma}
	\label{lem:S_nonempty}
Assume that 
$$w^0\ge \sum_{i=1}^k w^i\quad\textit{and}\quad  H^0\succeq H^i\textit{ for each }i=1,\ldots, k.
$$
(When $k=0$, the conditions say $w^0\ge 0$, $H^0\ge 0$.)
Then $\Omega_{k;1}(\bw)$ is non-empty. Furthermore,
for each $\alpha\in \Omega^1_{k;1}(\bw)$, the space $\pi^{-1}(\alpha)\subset \cS_{k;1}(\bw,\bH)$ is non-empty.
\end{lemma}

\begin{figure}[h]
	\includegraphics[]{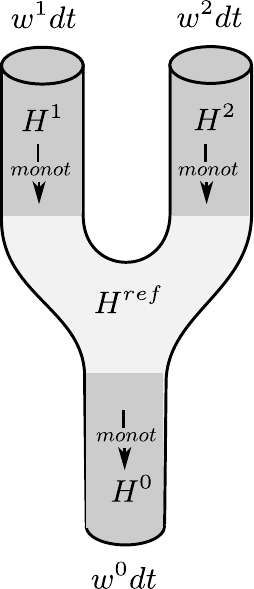}
	\caption{The Hamiltonian part of the Floer data from Lemma~\ref{lem:S_nonempty} on the pair-of-pants.}
	\label{fig:S_Nonempty}
\end{figure}

\begin{proof}
	It is well known that $\Omega_{k;1}(\bw)$ is non-empty: there exists a 1-form $\alpha$ satisfying $d\alpha\ge 0$ and having the desired behaviour at the cylindrical ends (see Lemma 16.1 of \cite{Ri13}). It remains to construct an $H$ such that $(\alpha,H)\in \cS_{k;1}(\bw,\bH)$; see Fig.~\ref{fig:S_Nonempty}. For this, consider a monotone homotopy from $H^i$ to $H^\textit{ref}$ on a cylindrical end near each input, and a monotone homotopy from $H^\textit{ref}$ to $H^0$ near the output. This defines $H$ on the union of three cylindrical ends. We assume that $\alpha$ is proportional to $dt$ on these cylindrical ends, so $d(H\otimes \alpha)\ge 0$ on them.  Outside of these regions, set $H\equiv H^\textit{ref}$. 
\end{proof}

\begin{lem}
	\label{lem:S_contr}
	Whenever $\cS_{k;1}(\bw,\bH)$ is non-empty, it is contractible.
\end{lem}

\begin{proof}
Consider the projection 
$\pi\co \cS_{k;1}(\bw,\bH)\to \Omega_{k;1}(\bw)$.
The base space $\Omega_{k;1}(\bw)$ is convex (whenever it is non-empty), hence contractible. Assuming the non-emptiness, each fibre of the projection is non-empty by Lemma~\ref{lem:S_nonempty}. Next, each fibre is obviously convex, hence contractible.  It follows that $\cS_{k;1}(\bw,\bH)$ is contractible.
\end{proof}

\begin{rem}
	\label{lem:s_c_std}
The space $\cS_c$ contains the usual continuation map data coming from monotone homotopies: namely, $\alpha=dt$ and $H\ge 0$ is non-decreasing in $s$.
\end{rem}

%

\subsection{Extrinsic continuation}

We assume that the reader is familiar with the algebra from \cite{VaThesis}; also see Section~\ref{sec:alg}.

\begin{defin}
	\label{dfn:c_e1e2}
	Let $\epsilon_1<\epsilon_2$.
	Let $\{H_i^1\}$, $\{H_i^0\}$ be cofinal families
	computing $SH^*_M(K,\epsilon_1)$ resp.~$SH^*_M(K,\epsilon_2)$, chosen so that $H_i^0\succeq H_i^1$ for each $i$.
	Consider two 1-rays from the diagram below.
	$$
	\begin{tikzcd}
	\ldots\ar[r]&CF^*(H_{i-1}^1)\ar[r]\ar[d]\ar[dr]&CF^*(H_{i}^1)\ar[r]\ar[d]\ar[dr]&CF^*(H_{i+1}^1)\ar[r]\ar[d]&\ldots\\
	\ldots\ar[r]&CF^*(H_{i-1}^0)\ar[r]&CF^*(H_{i}^0)\ar[r]&CF^*(H_{i+1}^0)\ar[r]&\ldots
	\end{tikzcd}
	$$
	The top ray is a Floer 1-ray for $SH^*_M(K,\epsilon_1)$, and the bottom ray is a Floer 1-ray for $SH^*_M(K,\epsilon_2)$.
	
	Construct a morphism between these 1-rays (i.e.~the vertical and diagonal arrows in the above diagram) by counting Floer solutions on the cylinder with perturbation data chosen consistently from the spaces $\cS_c(\bH)$. We use $\bH=(H^1_i,H^0_i)$ for the vertical arrows, and $\bH=(H^1_i,H^0_{i+1})$ for the diagonal arrows. (Recall that we use 1-dimensional families of data for the latter, as usual \cite{VaThesis}.)
	
	The induced $\Lambda_{\ge 0}$ module map 
	$$
	c_{\epsilon_1,\epsilon_2}\co SH^*_M(K,\epsilon_1)\to SH^*_M(K,\epsilon_2),
	$$
	is called an extrinsic continuation map. 
\end{defin}


\begin{defin}
	\label{dfn:c_e}
	Let $\epsilon>0$.
	Let $\{H_i^1\}$, $\{H_i^0\}$ be cofinal families
	computing $SH^*_M(K)$ resp.~$SH^*_M(K,\epsilon)$, chosen so that $H_i^0\geq H_i^1$ (the standard order). Note that the $H_i^1$ are not positive in this case.
	The $\Lambda_{\ge 0}$ module map
	$$
	c_{0,\epsilon_1}\co SH^*_M(K)\to SH^*_M(K,\epsilon)
	$$
	is defined analogously to \cite{Va18}, using continuation maps from monotone homotopies. (The perturbation data no longer live in the spaces $\cS_c(\bH)$, because the Hamiltonians are not everywhere positive.)
\end{defin}

\begin{lemma}
	\label{lem:c_commut}
	The maps from Definition~\ref{dfn:c_e1e2} and~\ref{dfn:c_e} are well-defined.	
For $0\le\epsilon_0<\epsilon_1<\epsilon_2$ there is a commutative diagram
$$
 \begin{tikzcd}
 SH^*_M(K,\epsilon_0)\arrow[r, "c_{\epsilon_0,\epsilon_1}"]\arrow[rd, "c_{\epsilon_0,\epsilon_2}"']
 &SH^*_M(K,\epsilon_1)\ar[d, "c_{\epsilon_1,\epsilon_2}"]
 \\
 &SH^*_M(K,\epsilon_2),
\end{tikzcd}
$$
where, if $\epsilon_0=0$, we formally put $SH^*_M(K,0)\coloneqq SH^*_M(K)$.
\end{lemma}

\begin{lemma}
	\label{lem:c_iso}
One has 
$$c_{0,\epsilon}=T^\epsilon\cdot f,$$ where $f$ is an isomorphism of $\Lambda_{\ge 0}$-modules.	
\end{lemma}

\begin{proof}
Let $\{H^1_i\}$ be a cofinal family for $SH^*_M(K)$, then $\{H^0_i\coloneqq H^1_i+\epsilon\}$ is a cofinal family for $SH^*_M(K,\epsilon)$. For each $i$, choose a monotone homotopy from $H^1_i$ to $H^0_i$ which has the following form:
$$
H_i^c(s,t)=H^1_i(t)+\rho(s)
$$
where $\rho(s)\co [0,1] \to [0,\epsilon]$ is a monotone function such that $\rho(s)= 0$ in a neighborhood of $0$ and $\rho(s)= \epsilon$ in a neighborhood of $1.$ Note that the addition of $\rho(s)$ does not change the Hamiltonian vector field term in the corresponding Floer equation. All rigid solutions to this continuation map are constant, but contribute with topological energy $T^\epsilon$.

Let $H_{i,i+1}^c(s,t)$ be a monotone homotopy from $H_i^1$ to $H_{i+1}^1$, and $c_{i,i+1}$ the induced continuation map between the Floer complexes. We choose the homotopy from
$H_i^0$ to $H_{i+1}^0$ (for the Floer 1-ray) to be $$H_{i,i+1}^c(s,t)+\epsilon.$$ Floer solutions for the these homotopies are precisely the same as for the unmodified homotopy $H_{i,i+1}^0(s,t)$. Abusing notation, we denote these continuation maps by  $c_{i,i+1}$ as well.

The diagram from Definition~\ref{dfn:c_e1e2} becomes the following.
$$
\begin{tikzcd}
\ldots \ar[r]&CF^*(H_{i}^1)\ar[r, "c_{i,i+1}"]\ar[d, "T^\epsilon\cdot \mathrm{Id}"']\ar[dr, "h_i"]&CF^*(H_{i+1}^1)\ar[r]\ar[d, "T^\epsilon\cdot \mathrm{Id}"]&\ldots\\
\ldots
\ar[r]&CF^*(H_{i}^0)\ar[r, "c_{i,i+1}"]&CF^*(H_{i+1}^0)\ar[r]&\ldots
\end{tikzcd}
$$
We can arrange for the homotopies $h_i$ to be identically zero here. In the general setup these would be defined by extending the given family of Hamiltonians on the boundary of a square to the entire square (as in Proposition 3.2.13 of \cite{Va18}) and then turning that into an interval family of Floer equations using a suitable Morse function on the square (e.g. Equation (3.2.0.5) of \cite{Va18}). For this application it is more convenient to directly go to the type of data that is used in the Floer equation (this will be done in the rest of this chapter as well, recall Remark \ref{remgluing}). To define the one parameter data interpolating between the two broken data, we first do gluing near the ends of the interval and note that the regions corresponding to the new supports of $\rho'$ and $\frac{d}{ds}H^c_{i+1}$ are still disjoint. Then to connect the two sides we simply move these supports off of each other by translating them in the (infinite) $s$ direction. When the supports intersect we are simply adding the two contributions. All Hamiltonians in the interior of this one parameter family are of the form: $$
H_{i,i+1}^c(s+\Delta,t)+\rho(s-\Delta),
$$ for some real number $\Delta$. The key point again is that changing Hamiltonians involved in the Floer equations by functions only depending on $s$ and $t$ do not change the equations. Using that $c_{i,i+1}$'s were defined by regular Floer data, we obtain that this parametrized family is also regular (transversality is clearly satisfied). But, notice that solutions of different members of this family of equations differ by translations in the $s$-direction. Hence, there is no rigid solution in this parametrized problem, proving the claim.

The upshot is that this map of 1-rays induces $T^\epsilon\cdot \mathrm{Id}$ at the level of relative symplectic cohomology, using the obvious generator-wise identification of all Floer complexes and hence the telescopes.
\end{proof}	

\begin{rem}
	It also holds that $c_{\epsilon_1,\epsilon_2}$ equals $T^{\epsilon_2-\epsilon_1}$ times an isomorphism. This follows from Lemmas~\ref{lem:c_commut} and~\ref{lem:c_iso}; alternatively, the proof of Lemma~\ref{lem:c_iso} adapts.
\end{rem}

\subsection{Variant extrinsic continuation}

We will now define an auxuliary map, $\tilde c_{\epsilon,2\epsilon}$, and show that it is chain homotopic to $c_{\epsilon,2\epsilon}$. We will later use $\tilde c_{\epsilon,2\epsilon}$ to prove the  unitality property.

\begin{defin}
	\label{dfn:c_tilde}
	Let $\epsilon>0$.
	Let $\{H_i^1\}$, $\{H_i^0\}$ be two different cofinal families
	computing $SH^*_M(K,\epsilon)$, chosen so that $H_i^0\succeq H_i^1$. Consider the 1-rays from the diagram below.
	$$
\begin{tikzcd}
\ldots\ar[r]&CF^*(H_{i-1}^1)\ar[r]\ar[d]\ar[dr]&CF^*(H_{i}^1)\ar[r]\ar[d]\ar[dr]&CF^*(H_{i+1}^1)\ar[r]\ar[d]&\ldots\\
\ldots\ar[r]&CF^*(2H_{i-1}^0)\ar[r]&CF^*(2H_{i}^0)\ar[r]&CF^*(2H_{i+1}^0)\ar[r]&\ldots
\end{tikzcd}
$$
The top ray is a Floer 1-ray for $SH^*(K,\epsilon)$, and the bottom is one for $SH^*(K,2\epsilon)$.
For the curves defining the bottom ray, we can assume that they solve Floer's equation 
using $\alpha=2dt$ as the 1-form and $H_i^0$ as the Hamiltonians, rather than $\alpha=dt$ and $2H_i^0$ as the Hamiltonians. This is tautological since $(2dt)\otimes H=dt\otimes (2H)$.

Construct a morphism between these 1-rays (i.e.~the vertical and diagonal arrows) by counting Floer solutions on the cylinder with perturbation data chosen consistently from the spaces $\cS_{\tilde c}(\bH)$ (Recall Table \ref{tab:weights}). That is, we use $\bH=(H^1_i,H^0_i)$ for the vertical arrows, and $\bH=(H^1_i,H^0_{i+1})$ for the diagonal arrows.

The induced $\Lambda_{\ge 0}$ module map 
	$$
	\tilde c_{\epsilon,2\epsilon}\co SH^*_M(K,\epsilon)\to SH^*_M(K,2\epsilon),
	$$
is called variant continuation map.
\end{defin}

\begin{lemma}
	\label{lem:c_versions}
	The maps $c_{\epsilon,2\epsilon}$ and $\tilde c_{\epsilon,2\epsilon}$ agree at the cohomology level.
\end{lemma}

\begin{proof} We start with two arbitrary Hamiltonians $H^1,H^0$ such that $H^0\succeq H^1>0$ and consider the spaces
	$$
	\cS_{\tilde c}= \cS_{1;1}(\bw,\bH),\quad \bw=(1,2),\quad \bH=(H^1,H^0)
	$$
	used in the definition of $\tilde c_{\epsilon,2\epsilon}$
	and
	$$
	\cS_c\coloneqq \cS_{1;1}(\bw,\bH),\quad \bw=(1,1),\quad \bH=(H^1,2H^0)
	$$
	used in the definition of $c_{\epsilon,2\epsilon}$.

		Let us fix a function $$
	\begin{array}{l}	\rho(s)\co \bR\to[1,2],\quad \partial_s\rho\ge 0,\\
	
	\rho(s)\equiv 1 \textit{ for }s\ll 0,\quad \rho(s)\equiv 2 \textit{ for }s\gg 0. 
\end{array}$$We define the following space
	$$
	\cC(H^1,H^0)=
	\{\rho(s-\Delta)\cdot dt\otimes H(s-\Delta)\}\subset\Omega^1(C)\otimes_{C^\infty(C)}  C^{\infty}(C\times M,\mathbb{R}).
	$$
	where:
	$$
	\begin{array}{l}
	\Delta\in\mathbb{R},\\
	H(s)\co M\times S^1\times \bR\to \bR,\quad  \partial_sH\ge 0,\\
	
	H(s)\equiv H^1 \textit{ for }s\ll 0,\quad H(s)\equiv H^0 \textit{ for }s\gg 0.
\end{array}
$$ 

Also consider the canonical map $q: \Omega^1(C)\times  C^\infty(C\times M)\to  \Omega^1(C)\otimes_{C^\infty(C)}  C^\infty(C\times M),$
$$(\alpha,H)\mapsto \alpha\otimes H.$$ Note that
$$
\cC(H^1,H^0)\subset q(\cS_c)\cap q(\cS_{\tilde{c}}),
$$ since one can write $(\rho(s-\Delta)\cdot dt)\otimes H(s-\Delta)=dt \otimes (\rho(s-\Delta)\cdot H(s-\Delta))$. 

Finally, let us note that if we have an element $(\beta, G)\in \cS_{1;1}((1,1),(H^1,H^0))$, then $(\beta, \rho G)\in \cS_{c}$ and $(\rho\beta, G)\in \cS_{\tilde{c}}$. Moreover, $q((\beta, \rho G))=q((\rho\beta, G))$.

Now let us go back our specific goal. We fix the same cofinal families $\{H_i\}$ and $\{H_i\}$ for the set-up described in Definition \ref{dfn:c_tilde} (and also monotone homotopies as in the definition of acceleration data). Then, we want to construct $c_{\epsilon,2\epsilon}$ and $\tilde c_{\epsilon,2\epsilon}$ as in Definition \ref{dfn:c_e1e2} (with $(\epsilon_1,\epsilon_2) =(\epsilon, 2\epsilon)$ and cofinal families $\{H_i\}$ and $\{2H_i\}$) and Definition \ref{dfn:c_tilde}, respectively, in such a way that after applying the map $q$, the data that is used to define them are the same. This would finish the proof.

Let $(dt, G_i(s,t):=H_i(t))\in (\cS_{1;1}((1,1),(H_i,H_i))$ and set up the vertical continuation maps using $(dt, \rho G_i)$ and $(\rho dt, G_i).$ Next, we need to define the homotopies. As usual, first glue near the boundaries of the interval. Observe that the $q$-projections of both perturbation data we obtain live in $\cC(H_i,H_{i+1})$. Remembering the particular form of these data, we can connect the two end points inside $\cC(H_i,H_{i+1})$ by shifting the supports. It is easy to see that the paths we obtained can also be lifted to the relevant $\cS_{c}$ and $\cS_{\tilde{c}}$ spaces.

\end{proof}

\subsection{Product and unit on raised cohomology}

\begin{defin}
	\label{dfn:prod_e}
	Let $\epsilon>0$.
	Let $\{H_i^1\}$, $\{H_i^2\}$, $\{H_i^0\}$ be three different cofinal families
	computing $SH^*_M(K,\epsilon)$, chosen so that $H_i^0\succeq H_i^1,H^2_i$
	for all $i$.
	They give rise to three Floer 1-rays.  Consider two 1-rays from the diagram below.
$$
\begin{tikzcd}
\ldots \ar[r]\ar[dr]&CF^*(H_{i-1}^2)\otimes CF^*(H_{i-1}^1)\ar[r]\ar[d]\ar[dr]&CF^*(H_{i}^2)\otimes CF^*(H_{i}^1)\ar[r]\ar[d]\ar[dr]& \ldots\\
\ldots \ar[r]& CF^*(2H_{i-1}^0)\ar[r]&CF^*(2H_{i}^0)\ar[r]& \ldots
\end{tikzcd}
$$
The bottom ray is the  Floer  1-ray corresponding to $\{2H_i^0\}$.	The top ray is the slice-wise tensor product of the Floer 1-rays corresponding to $\{H_i^1\}$, $\{H_i^2\}$.

Construct  a morphism between these rays  by counting  Floer solutions on pairs-of-pants where the perturbation data are consistently chosen from 
$\cS_{p}(\bH)$. As in Definition~\ref{dfn:c_tilde}, we  use the fact that $(2dt)\otimes H=dt\otimes (2H)$ to treat the  output orbits of the pairs-of-pants as elements of $CF^*(2H_i^0)$.

Using this morphism, induce a $\Lambda_{\ge 0}$ module map called the product:
	$$
	*_{\epsilon}\co SH^*_M(K,\epsilon)\otimes SH^*_M(K,\epsilon) \to SH^*_M(K,2 \epsilon)
	$$
	as explained in Subsection~\ref{subsec:tensor}.
\end{defin}

\begin{rem}In this paper, we do not bother with proving that this product is independent of the choices, as we have no use for it. 
\end{rem}

\begin{defin}
	\label{dfn:unit_e}
	Let $\epsilon>0$, and $\{H_i^2\}$, $H_i^2\ge 0$, be a cofinal family
	computing $SH^*_M(K,\epsilon)$.
	
	Consider the 1-ray $U$ from Subsection~\ref{subsec:algebra_units} and the Floer 1-ray for $\{H_i^2\}$. They are the rows of the  diagram below. 
	$$
	\begin{tikzcd}
	\ldots\ar[r]&\Lambda_{\ge 0}\ar[r]\ar[d]\ar[dr]&\Lambda_{\ge 0}\ar[r]\ar[d]\ar[dr]&\Lambda_{\ge 0}\ar[r]\ar[d]&\ldots\\
	\ldots\ar[r]&CF^*(H_{i-1}^2)\ar[r]&CF^*(H_{i}^2)\ar[r]&CF^*(H_{i+1}^2)\ar[r]&\ldots
	\end{tikzcd}
	$$
	Construct  a morphism between these rays by counting  Floer solutions over $\bC P^1\setminus\{z_0\}$ where the perturbation data are consistently chosen from 
	$\cS_{u}(\bH)$. 
	
	Using this morphism, define
	the elements
	$$
	1_{K,\epsilon}\in SH^*_M(K,\epsilon)
	$$
	as in Subsection~\ref{subsec:algebra_units}. 
\end{defin}

We do need this element to be well-defined.

\begin{rem}
	\label{rem:unit_hard}Note that we need to define a product on relative symplectic homology using structures at the chain level as it is not in general not determined by homology level Hamiltonian Floer theoretic data. This is of course because of the completion step.
It is possible to define the Floer product directly on the usual relative symplectic cohomology: $SH^*_M(K)\otimes SH^*_M(K) \to SH^*_M(K)$. But, for example we have that $SH^*_M(M)=H(M,\Lambda_{\geq 0})\otimes \Lambda_{> 0}$ (Section 3.3.3 of \cite{VaThesis}). This means that the unit can only be expected to exist after tensoring with $\Lambda$. Our solution to resolve this was to introduce raised cohomology and first define the unit as an element there, which let us keep working over the Novikov ring up to this point.

\end{rem}

\begin{lemma}
	\label{lem:prod_and_cont}
	The  elements $1_{K,\epsilon}\in SH^*_M(K,\epsilon)$ do not depend on the choices. Moreover, they are compatible with the maps $c_{\epsilon_1,\epsilon_2}$ for $\epsilon_1<\epsilon_2$, i.e.
	$$
	1_{K,\epsilon_2}=c_{\epsilon_1,\epsilon_2}(1_{K,\epsilon_1}).
	$$
	
\end{lemma}

\begin{proof}
	This is again a standard application of Lemmas~\ref{lem:S_nonempty} and~\ref{lem:S_contr}, along with Lemma \ref{lemstrictunit}.
\end{proof}	


\begin{lem}
	\label{lem:unit_real_cont} The map of $1$-rays from Definition~\ref{dfn:unit_e} (units) is a realisation of the map of $1$-rays from Definition~\ref{dfn:c_tilde} (variant continuation maps) via the map of $1$-rays from Definition~\ref{dfn:prod_e} (product) in the sense of Definition \ref{defrealization} of Subsection~\ref{subsec:algebra_units}.
\end{lem}

\begin{proof}
	There are natural gluing maps of spaces  of perturbation data
	$$
	\#_\rho\co \cS_u\sqcup \cS_p\to \cS_{\tilde c},\quad \rho\gg 0. 
	$$
	Using the gluing maps and the contractibility of these spaces, one can fill in the arrows of the diagram for the desired homotopy.
\end{proof}

\begin{cor}
	\label{cor:1_e_c}
At the cohomology level, product with the unit $-*_\epsilon 1_\epsilon$ equals the extrinsic continuation map $c_{\epsilon,2\epsilon}$.
\end{cor}

\begin{proof}
	This follows from Lemma~\ref{lem:unit_real_cont}, Lemma~\ref{lem:c_versions} and Lemma \ref{lem:realis_h}.
\end{proof}

\subsection{Product and unit on relative cohomology}

From now on, let us consider the symplectic  cohomologies over $\Lambda$: $$SH^*_M(K,\epsilon,\Lambda)\coloneqq SH^*_M(K,\epsilon)\otimes \Lambda.$$

\begin{defin}
Define the product 
$$*\co SH^*_M(K,\Lambda)\otimes SH^*_M(K,\Lambda) \to SH^*_M(K,\Lambda)$$ as the unique map making the diagram below commutative:
$$
\begin{tikzcd}
 SH^*_M(K,\epsilon,\Lambda)\otimes SH^*_M(K,\epsilon,\Lambda)
\ar[r, "*_\epsilon"]
&
 SH^*_M(K,2\epsilon,\Lambda)
\\
 SH^*_M(K,\Lambda)\otimes SH^*_M(K,\Lambda)
\ar[r, "*"]
\ar[u, "c_{0,\epsilon}\otimes c_{0,\epsilon}"]
&
SH^*_M(K,\Lambda)
\ar[u, "c_{0,2\epsilon}"']
\end{tikzcd}
$$
Recall that the vertical arrows are isomorphisms over $\Lambda$ by Lemma~\ref{lem:c_iso}.

Define the unit $1_K\in SH^*_M(K,\Lambda)$ as $$1= c_{0,\epsilon}^{-1}(1_{K,\epsilon}).$$
Again, recall that $c_{0,\epsilon}$ is  an isomorphism over $\Lambda$ by Lemma~\ref{lem:c_iso}.
\end{defin}

\begin{lem}
The element $1_K$ are well-defined. Moreover, $1_K$ is a 2-sided unit for $*$. 
\end{lem}

\begin{proof}
	Let us check that $1_K$ is a unit. Consider the following diagram: 
$$
\begin{tikzcd}
SH^*_M(K,\epsilon,\Lambda)\otimes \Lambda\langle  1_{K,\epsilon}\rangle
\ar[r, "-*_\epsilon 1_{K,\epsilon}", bend left]
\ar[r, "c_{\epsilon,2\epsilon}", bend right]
&
SH^*_M(K,2\epsilon,\Lambda)
\\
SH^*_M(K,\Lambda)\otimes  \Lambda\langle  1_K  \rangle
\ar[r, "-*1_K"]
\ar[u, "c_{0,\epsilon}\otimes c_{0,\epsilon}"]
&
SH^*_M(K,\Lambda)
\ar[u, "c_{0,2\epsilon}"']
\end{tikzcd}
$$
The outer square  is commutative by the definition of units and products. The slit is commutative by Corollary~\ref{cor:1_e_c}. Looking at the inner square 
and Lemma~\ref{lem:c_commut} implies that $-*1$ is the identity.
\end{proof}

\begin{cor}
	\label{lem:1_zero_sh_zero}
If $1_K=0\in SH^*_M(K)$, then $SH^*_M(K)=0$.
\end{cor}

\begin{proof}
	This follows from Corollary~\ref{cor:real_iso}.
\end{proof}	

\begin{rem}
Notice that this statement does not make any reference to the product, but its proof relied on the existence of a product structure.
\end{rem}

\begin{lem}
Restriction maps respect units $1_K\in SH^*_M(K)$. 
\end{lem}	

\begin{proof}
	It is straightforward to define restriction maps $r_\epsilon\co SH^*_M(K,\epsilon)\to SH^*_M(K',\epsilon)$ for $K\subset K'$ and prove that the following diagram is commutative:
$$
\begin{tikzcd}
SH^*_M(K',\epsilon)
\ar[r, "r_\epsilon"]
&
SH^*_M(K,\epsilon)
\\
SH^*_M(K')
\ar[r, "r"]
\ar[u, "c_{0,\epsilon}"]
&
SH^*_M(K)
\ar[u, "c_{0,\epsilon}"]
\end{tikzcd}
$$
The bottom row is the usual restriction map.
It is also straightforward to show that $r_\epsilon$ respects the elements $1_{K,\epsilon}$ using Lemma \ref{lemstrictunit}. The lemma follows.	
\end{proof}

\section{Relative Lagrangian Floer theory and closed-open maps}
\label{sec:lag}

\subsection{Lagrangian Floer theory}
Let $L\subset M$ be a tautologically unobstructed oriented Lagrangian submanifold with a relative Pin structure.

As explained in Section~\ref{sec:rel_sh}, for every compact set $K\subset M$ there exists a Lagrangian version of relative symplectic cohomology, $HF^*(L,K)$.

Completely analogously to the previous section, one defines the raised version $HF^*(L,K,\epsilon)$ as the homology of the completed telescope of
$$
\begin{tikzcd}
\ldots\ar[r]&CF^*(L,H_{i-1})\ar[r]&CF^*(L,H_{i})\ar[r]&CF^*(L,H_{i+1})\ar[r]&\ldots
\end{tikzcd}
$$
Here $CF^*(L,H)$ is the usual Floer complex of $L$ with itself using the Hamiltonian $H$, and the $H_i$ are an increasing cofinal family Hamiltonians as in the beginning of Section~\ref{sec:prod_unit}.
Also analogously to Section~\ref{sec:prod_unit},
one defines extrinsic continuation maps, product, and the unit 
$$
1_{K,L,\epsilon}\in HF^*(L,K,\epsilon).
$$ Finally, one defines the product  and unit on $HF^*(L,K)$ by repeating the formal trick from Section~\ref{sec:prod_unit}. They satisfy obvious versions of the properties from that section.

The only detail which is slightly different is the definition of the  spaces of perturbation data.
Let $C$ be the disk with $k\in\{0,1,2\}$ input punctures and one output puncture. One uses the spaces
$$
\label{eq:S_L}
\cS^{op}_{k;1}(\bw,\bH)=\{
(\alpha, H)\}
\subset \Omega^1(C)\times  C^\infty(C\times M)
$$
where $\alpha,H$ with the following properties:
$$
\left\{
\begin{array}{l}
d\alpha\ge 0,\ H\ge 0,\\
d_{C\times M}(H\alpha)\mid_{C\times\{x\}}\ge 0, \text{ for every }x\in M\\
\alpha\equiv w^idt\text { on strip-like ends,}\\
H\equiv H^i(t)\text { on strip-like ends,}\\
T(\partial C)\subset \ker \alpha.
\end{array}
\right.
$$

The last property has no analogue in the closed-string case. It is required to guarantee that the topological energy of a Floer strip is indeed topological.
Analogues of Lemmas~\ref{lem:S_nonempty} and \ref{lem:S_contr} on non-emptiness and contractibility of these spaces hold under similar conditions. 

\begin{figure}[h]
	\includegraphics[]{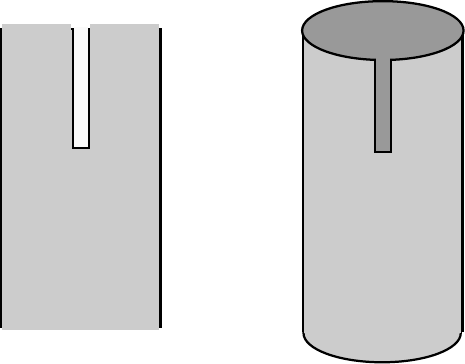}
	\caption{The Riemann surfaces used to define the product structure for the open string version (left) and the closed-open maps (right).}
	\label{fig:slit}
\end{figure}

One point worth a clarification is why, given $w^0\ge \sum_{i=1}^k w_i$, there exists at least one $\alpha$ with $d\alpha\ge 0$ and $T(\partial C)\subset \ker\alpha$, compare with the beginning of the proof of Lemma~\ref{lem:S_nonempty}. Assume, for instance, $k=2$ and $w^0= w_1+w_2$ (it is easy to upgrade the example to the case of the inequality). Consider the domain $C$, the disk with three boundary punctures, conformally represented in Figure~\ref{fig:slit} by a strip of width $w^0$ with an extra slit dividing the widths into $w^1$ and $w^2$. The form $\alpha=dt$ in this representation has the desired properties: it restricts to $w^idt$ on strip-like ends (recall that a strip-like end, by definition, conformally reparametrises a neighboirhood of a puncture to a strip of width one), and $\alpha$ vanishes on $T(\partial C)$, including the tangent directions to the slit.

\subsection{Closed-open maps}
The last ingredient is to relate the closed- and open-string theories by the closed-open map, with the aim of establishing Proposition~\ref{prop:zero_unit_hf_sh}.

For this, let $C$ be the unit disk with one boundary puncture (considered as output) and one interior puncture (considered as input) modulo automorphisms. Equip $C$ with a strip-like end at the boundary puncture, and cylindrical end at the interior puncture. One can assume that the punctures are at points 1 and 0, respectively. 
Define 
\begin{equation}
\label{eq:S_co}
\cS_{1;0;1}(\bw,\bH)=\{
(\alpha, H)\}
\subset \Omega^1(C)\times  C^\infty(C\times M)
 \textit{ where }
\bw=(w^1,w^0),\ \bH=(H^1,H^0).
\end{equation}
to be the subset consisting of all $\alpha,H$ with the following properties:
$$
\left\{
\begin{array}{l}
d\alpha\ge 0,\ H\ge 0,\\
d_{C\times M}(H\alpha)\mid_{C\times\{x\}}\ge 0, \text{ for every }x\in M\\
\alpha\equiv w^1dt\text { on the cylindrical  end,}\\
H\equiv H^1(t)\text { on the cylindrical  end,}\\
\alpha\equiv w^0dt\text { on the strip-like  end,}\\
H\equiv H^0(t)\text { on the strip-like end,}\\
T(\partial C)\subset \ker \alpha.
\end{array}
\right.
$$

Assume that $w^0\ge w^1$ and $H^0\succeq H^1$, then again $\cS_{1;0;1}(\bw,\bH)$ is non-empty and contractible.
To set up closed-open maps between raised cohomologies, we use perturbation spaces of the form
$$
\cS_{co}(\bH)=\cS_{1;0;1}(\bw,\bH)\quad\textit{with}\quad  w^1=w^0=1.
$$

Let $\epsilon>0$,
and $\{H_i^1\}$, $\{H_i^0\}$ be two cofinal families
for $SH^*_M(K,\epsilon)$ chosen so that $H_i^0\succeq H_i^0$ for each $i$.
Consider two Floer 1-rays from the diagram below, and construct the morphism between them  using  curves on the disk $C$ with one interior input and one boundary output as above, with perturbation  data chosen consistently from $\cS_{co}(\bH)$ using $\bH=(H_i^1,H_i^0)$  and $(H_i^1,H_{i+1}^0)$.
$$
\begin{tikzcd}
\ldots\ar[r]&CF^*(H_{i-1}^1)\ar[r]\ar[d]\ar[dr]&CF^*(H_{i}^1)\ar[r]\ar[d]\ar[dr]&CF^*(H_{i+1}^1)\ar[r]\ar[d]&\ldots\\
\ldots\ar[r]&CF^*(L,H_{i-1}^0)\ar[r]&CF^*(L,H_{i}^0)\ar[r]&CF^*(L,H_{i+1}^0)\ar[r]&\ldots
\end{tikzcd}
$$
This gives rise to $\Lambda_{\ge 0}$ module map
$$
\cC\cO_\epsilon\co SH^*_M(K,\epsilon)\to HF^*(L,K,\epsilon).
$$
One formally defines
$$
\cC\cO\co  SH^*_M(K,\Lambda)\to HF^*(L,K,\Lambda)
$$
by requiring the following diagram to commute (recall that the vertical arrows are isomorphisms over $\Lambda$:

$$
\begin{tikzcd}
SH^*_M(K,\epsilon)\otimes \Lambda \arrow[r, "\cC\cO_\epsilon"]& HF^*(L,K,\epsilon)\otimes \Lambda \\
SH^*_M(K)\otimes \Lambda \arrow[r, "\cC\cO"]\arrow[u,"c_{0,\epsilon}"]& HF^*(L,K)\otimes \Lambda\arrow[u,"c_{0,\epsilon}"]
\end{tikzcd}
$$

\begin{lem}
	\label{lem:co_unit}
The closed-open map satisfy
$$
\cC\cO(1_{K,\epsilon})=1_{K,L,\epsilon},\quad \cC\cO(1_K)=1_{K,L}.
$$
\end{lem}

\begin{proof}
Let $\cC$ be a Floer 1-ray for $SH^*_M(K,\epsilon)$, $\cC'$ be one for $HF^*(L,K,\epsilon)$, and $U$ the 1-ray from Subsection~\ref{subsec:algebra_units}.
One constructs a homotopy
$$
\begin{tikzcd}
U\arrow[r, "1_{K,\epsilon}"]\arrow[rr, bend right, "1_{K,L,\epsilon}"'] &\cC\arrow[r, "\cC\cO"] &\cC'
\end{tikzcd}
$$
by counting  Floer solutions of appropriate continuation maps. This uses the fact that there are natural gluing maps
$$
\#_\rho\co \cS_{u}\sqcup \cS_{co}\to \cS^{op}_{u},\quad \rho\gg 0. 
$$
The perturbation spaces $\cS_{u}$ and $\cS_{co}$ were introduced above (they are used to define the unit and the closed-open map), and  $\cS^{op}_{u}(\bH)\coloneqq \cS^{op}_{0;1}(\bw,\bH)$ with $\bw=w^0=1$ are the perturbation spaces from the definition of unit in Lagrangian Floer cohomology. Hence, the first equality follows from Lemma \ref{lemstrictunit}.

The second equality from the lemma follows from the first one, and the definitions of $1_K$, $1_{K,L}$.
\end{proof}

\appendix
\section{Liouville complements}

%
%

\begin{proof}[Proof of Proposition \ref{prop:giroux_skeleton}]
First of all, note that McLean's Proposition 6.17 reduces the problem to showing that there is a primitive $\theta$ defined on $M-D$ such that the relative cohomology class in $H^2(M,M-D)$ defined by $(\omega,\theta)$ is equal to $\sum \frac{w_i}{c}[D_i]$. It suffices to prove this for $c=1$, so let us assume that.

Let $(L,h,\nabla)$ be the pre-quantization complex line bundle for $(M,\omega)$. More precisely, $h$ is a Hermitian metric on $L$ and $\nabla$ is a compatible connection such that the curvature $2$-form of $\nabla$ is equal to $\omega$.  

Let $O_M(D)$ be the complex line bundle associated to the SC symplectic divisor $D$ as explained in the discussion near Equations (6) and (7) of \cite{tehrani}. Then $O_M(D)$ and $L$ are isomorphic as complex line bundles because they have the same first Chern class.  This in particular shows that $L$ has a section which vanishes precisely along $D$ with multiplicity $w_i$ on $D_i$ by taking the tensor product of the ``defining" sections of $O(D_i)$. Let us call this section $s$. 

Now let $P\to M$ be the $U(1)$-bundle associated to $L$. Then we have a connection one form $\theta'$ on $P$ which is a primitive of the pullback of $\omega$ to $P$. We can construct a section $s'$ of $P\to M$ over $M-D$ using $s$, by $s':=s/\abs{s}$. Pulling back $\theta'$ by $s'$, we obtain a primitive $\theta$ of $\omega$ on $M-D$, which satisfies the desired conditions.

\end{proof}

\bibliographystyle{plain}
\bibliography{Symp_bib}

\end{document}